\documentclass[11pt,reqno]{amsart}

\usepackage{amsrefs,amsmath,amsthm,amssymb,graphicx,enumitem,stmaryrd,hyperref}
\usepackage{tikz}
\usetikzlibrary{arrows,automata,positioning}

\usepackage{url}
\usepackage{hyperref}

\usepackage{caption}

\usepackage{algorithm}
\usepackage{algorithmic}

\usepackage{mathtools}
\mathtoolsset{showonlyrefs=true}

\numberwithin{equation}{section}

\theoremstyle{definition}
\newtheorem{defi}{Definition}[section]

\newtheorem*{nota*}{Notation}

\theoremstyle{plain}
\newtheorem{lemm}[defi]{Lemma}

\newtheorem{prop}[defi]{Proposition}
\newtheorem{theo}[defi]{Theorem}
\newtheorem*{theo*}{Theorem}
\newtheorem{coro}[defi]{Corollary}
\newtheorem{conj}[defi]{Conjecture}

\theoremstyle{remark}
\newtheorem{rema}[defi]{Remark}
\newtheorem*{rema*}{Remark}
\newtheorem{exam}[defi]{Example}
\newtheorem*{exam*}{Example}

\renewcommand{\subset}{\subseteq}
\renewcommand{\supset}{\supseteq}

\newcommand{\cC}{\mathcal{C}}
\newcommand{\cN}{\mathcal{N}}
\newcommand{\cD}{\mathcal{D}}

\newcommand{\bbZ}{\mathbb{Z}}
\newcommand{\bbQ}{\mathbb{Q}}

\newcommand{\de}{\delta}
\newcommand{\De}{\Delta}
\newcommand{\la}{\lambda}
\newcommand{\varep}{\varepsilon}
\newcommand{\vep}{\varepsilon}
\newcommand{\DS}{\displaystyle}
\newcommand{\epcl}{\mathcal{E}}
\newcommand{\cl}{\epcl}

\newcommand{\lra}{\longrightarrow}

\newcommand{\Lra}{\Longrightarrow}

\newcommand{\Llra}{\Longleftrightarrow}

\newcommand{\sm}{\setminus}
\newcommand{\relmiddle}[1]{\mathrel{}\middle#1\mathrel{}}
\renewcommand{\l}{\ell}
\newcommand{\len}{\ell}

\newcommand{\Na}[1]{\textup{\textsf{(N{#1})}}}
\newcommand{\Naa}[1]{\textup{\textsf{(N{#1}$'$)}}}
\newcommand{\Naak}{\textup{\textsf{(N{6}$'_k$)}}}
\newcommand{\Pia}[1]{\textup{\textsf{(P{#1}$_j$)}}}

\newcommand{\PTsim}{\overset{\text{\textup{PT}}}{\sim}}
\newcommand{\simPT}{\PTsim}

\newcommand{\Par}{\mathsf{Par}}
\newcommand{\hPar}{\widehat{\Par}}
\newcommand{\Parr}[1]{\Par_{#1}}

\DeclareMathOperator{\twt}{wt}
\DeclareMathOperator{\id}{id}

\newcommand{\rsh}{\phi_{+}}
\newcommand{\lsh}{\phi_{-}}

\newcommand{\pst}{\bullet}

\newcommand{\AL}{\Sigma}
\newcommand{\AUT}{M}
\newcommand{\FB}{X}
\newcommand{\emptypar}{\boldsymbol{\emptyset}}

\DeclareMathOperator{\Base}{\mathsf{minimal}}
\newcommand{\OF}{V}

\newcommand{\sums}[1]{\sum_{\substack{#1}}}
\newcommand{\Sinf}{\mathsf{Seq}}
\newcommand{\Seq}{\mathsf{Seq}}
\newcommand{\spa}{s}

\newcommand{\Avoid}[3]{\mathsf{avoid}(#1,#2,#3)}
\newcommand{\vi}{\vec{i}}
\newcommand{\transpose}[1]{\vphantom{#1}^{t}\hspace{-.5mm}#1}
\newcommand{\hatt}{\widehat}

\newcommand{\comp}[1]{{#1}^{\wedge}}

\newcommand{\refEuler}{\textup{(A)}}
\newcommand{\refEulerr}{\textup{(B)}}

\newcommand{\RR}{Rogers--Ramanujan}
\newcommand{\RRs}{RR}

\title{A proof of conjectured partition identities of Nandi}
\date{\today}

\author{Motoki Takigiku}
\address{Graduate School of Natural Science and Technology, Okayama University, Okayama 700-8530, Japan}
\email{takigiku@okayama-u.ac.jp}

\author{Shunsuke Tsuchioka}
\address{Department of Mathematical and Computing Sciences, Tokyo Institute of Technology, Tokyo 152-8551, Japan}
\email{tshun@kurims.kyoto-u.ac.jp}

\begin{document}
\maketitle
\begin{abstract}
We generalize the theory of linked partition ideals due to Andrews
    using finite automata in formal language theory and apply it to
    prove three Rogers--Ramanujan type identities of modulo 14
    that were posed by Nandi through vertex operator theoretic construction of
    the level 4 standard modules of the affine Lie algebra $A^{(2)}_{2}$.
\end{abstract}

\section{Introduction}\label{sect:intro}

\subsection{\RR{} type identities}\label{sect:intro:intro}

A partition of an integer $n\ge 0$ is 
a weakly decreasing sequence of positive integers (called parts) whose sum is $n$.
Let $\Par$ denote the set of all partitions.

The celebrated \emph{Rogers--Ramanujan 
identities} are stated as
\begin{equation}\label{eq:RR:PT}
    \begin{minipage}{.90\textwidth}
        ``partitions of $n$ 
        whose parts are $\ge i$ and mutually differ by at least 2
        are equinumerous to partitions of $n$ into parts $\equiv \pm i\ (\bmod\,5)$''
    \end{minipage}
\end{equation}
(for $i=1,2$) in the
combinatorial form,
and their $q$-series versions are
\begin{equation}\label{eq:RR:q}
    \sum_{n\ge 0} \frac{q^{n^2}}{(q;q)_n}
    = \frac{1}{(q,q^4;q^5)_\infty},
    \quad
    \sum_{n\ge 0} \frac{q^{n^2+n}}{(q;q)_n}
    = \frac{1}{(q^2,q^3;q^5)_\infty},
\end{equation}
where we write
for $n\in\bbZ_{\ge 0} \sqcup \{\infty\}$
\[
    (a;q)_n := \prod_{0\le j<n} (1-aq^j),
    \quad
    (a_1,\dots,a_k;q)_{n}
    := (a_1;q)_n \cdots (a_k;q)_n.
\]

These identities \eqref{eq:RR:PT} and \eqref{eq:RR:q}
have
a number of generalizations,
often called \emph{Rogers--Ramanujan type} 
(\emph{\RRs{} type} for short)
identities,
arising from various motivations (see e.g.,\,\cite{MR0557013}*{\S7}, \cite{MR3752624}).
In particular,
generalizations of \eqref{eq:RR:PT}
are 
called 
RR type partition identities,
which are theorems of the form $\cC\simPT\cD$ for $\cC, \cD\subset\Par$ 
(\cite{MR387178}*{Definition 3}), meaning that
partitions of $n$ in $\cC$ are equinumerous to those in $\cD$ for all $n\ge0$,
where most commonly $\cC$ (resp. $\cD$) is given by 
``difference conditions'' (resp. ``mod conditions'') on parts.

\subsection{Algorithmic  derivation of $q$-difference equations}
\label{sect:intro:auto_qdiff}
A common strategy to prove
an \RRs{} type partition identity
is as follows (cf.\,\cite{MR387178}*{p.1037}):
starting with the set $\cC\subset\Par$ 
(given by ``difference conditions'' in the statement of the identity),

\begin{itemize}
    \item [(Step 1)] 
        Find a $q$-difference equation for 
        the generating function
        \begin{equation}\label{eq:defi:f_C}
            f_\cC(x,q) := \sum_{\la\in\cC} x^{\len(\la)}q^{|\la|},
        \end{equation}
where
$|\la|:=\sum_{i=1}^{\ell} \la_i\,(=\sum_{i\geq 1}i\,m_i(\lambda))$ and
$\len(\la):=\ell\,(=\sum_{i\geq 1}m_i(\lambda))$
for
$\la=(\la_1,\la_2,\cdots,\la_\ell)\in\Par$ and $m_i(\la) := \#\{j\mid \la_j=i\}$ for $i\ge 1$.

    \item [(Step 2)]
        Solve the equation and find a $q$-series expression for 
        $f_{\cC}(1,q)=\sum_{\la\in\cC}q^{|\la|}$.
    \item [(Step 3)]
        Use $q$-series formulas 
        and show that $f_{\cC}(1,q)$ is equal to the desired infinite product
        corresponding to $\cD$.
\end{itemize}

The aim of this paper is to give a proof (following these steps) for
three conjectural \RRs{} type partition identities 
(see \S\ref{sect:intro:statement})
posed by Nandi \cite{NandiPhD}
and, to do so,
give an extention
(by using \emph{finite automata} in \emph{formal language theory})
of the theory of \emph{linked partition ideals}
introduced by Andrews \cite{MR387178},\cite{MR0557013}*{\S 8},
which provides in many cases an algorithmic derivation in the Step 1 above (see \S\ref{sect:intro:outline} for more details).

\subsection{Nandi's conjectures}
\label{sect:intro:statement}

The \RR{} identities 
was one of the motivations for inventing
vertex operators;
it started from Lepowsky--Milne's observation \cite{MR501091},
which led to Lepowsky--Wilson's proof for \RR{} identities
\cites{MR638674,MR663415,MR752821}
by constructing
bases of
the vacuum spaces $\Omega(V(\la))$ for
the standard modules $V(\la)$ of
the affine Lie algebra $A^{(1)}_{1}$
associated with the level 3 dominant integral weights $\la$,
using certain vertex operators called \emph{$Z$-operators}.
Moreover, 
Andrews--Gordon's \cites{MR351985,MR123484}
and Andrews--Bressoud's \cites{MR541344,MR556608}
generalizations of the \RR{} identities
can be interpreted and proved 
via similar constructions for the level $\ge4$ standard modules 
of $A^{(1)}_{1}$\cites{MR752821,MR782227,MR888628}.

It is therefore natural to expect that
there should exist an \RRs{} type identity
corresponding to
any given affine Lie type and a dominant integral weight.
As a first step 
beyond the case $A^{(1)}_{1}$, 
Capparelli \cite{MR2637254} investigated 
the structure of the level 3 standard modules of
the affine Lie algebra $A^{(2)}_{2}$ via $Z$-operators,
yielding some conjectural partition identities
(which were later proved by
\cites{MR1284057,MR1333389,MR1364151,1912.03689,2006.02630}
etc.).
As a next step, Nandi \cite{NandiPhD} studied the level 4 standard modules of $A_{2}^{(2)}$ via $Z$-operators
and conjectured some partition identities (Conjecture \ref{theo:Nandi_conj}).
For higher levels, see e.g.,\,\cite{MR3773946},\cite{2006.02630}*{\S1.4}.

\begin{defi}\label{defi:match}
    For a finite sequence
    $\vec{j}=(j_1,\dots,j_n)$ 
    (which we assume to be nonempty for simplicity, i.e., $n>0$) and
    a (finite or infinite) sequence 
    $\vec{i}=(i_1,\dots,i_N)$ or
    $\vec{i}=(i_1,i_2,\dots)$,
    we say that,
    letting $\mathsf{len}(\vi):=N(\geq 0)$ or $\infty$ respectively, 
    \begin{itemize}
        \item 
        $\vec{i}$ 
        \emph{matches} $\vec{j}$ if
        $(i_{k+1},i_{k+2},\dots,i_{k+n})=(j_1,j_2,\dots,j_n)$ 
        for some $0\le k\le \mathsf{len}(\vi)-n$,
        \item
        $\vec{i}$ \emph{begins with} $\vec{j}$ if
        $n\le \mathsf{len}(\vi)$ and
        $(i_1,i_2,\dots,i_{n})=(j_1,j_2,\dots,j_n)$.
    \end{itemize}
\end{defi}

\begin{conj}[Nandi {\cite{NandiPhD}*{\S 8.1}. See also \cite{MR3752624}*{Conjecture 5.5, 5.6, 5.7}}]\label{theo:Nandi_conj}
Let $\cN$ denote the set of partitions $\la$ satisfying the conditions \Na{1}-\Na{6}:
\begin{itemize}
    \item[\Na{1}] For all $1\le i\le \len(\la)-1$, $\la_i - \la_{i+1} \neq 1$,
    \item[\Na{2}] For all $1\le i\le \len(\la)-2$, $\la_i - \la_{i+2} \ge 3$,
    \item[\Na{3}] For all $1\le i\le \len(\la)-2$, $\la_i - \la_{i+2} = 3 \implies \la_i \neq \la_{i+1}$,
    \item[\Na{4}] For all $1\le i\le \len(\la)-2$, 
        $\la_i - \la_{i+2} = 3 \text{ and } 2\nmid \la_i 
            \implies \la_{i+1} \neq \la_{i+2}$,
    \item[\Na{5}] For all $1\le i\le \len(\la)-2$, \\
        $\la_i - \la_{i+2} = 4 \text{ and } 2\nmid \la_i 
            \implies \la_{i} \neq \la_{i+1} \text{ and } \la_{i+1}\neq \la_{i+2}$,
    \item[\Na{6}] 
        $(\la_1-\la_2, \la_2-\la_3,\cdots,\la_{\ell(\la)-1}-\la_{\ell(\la)})$
        does not match $(3,2^*,3,0)$.
        Here $2^*$ denotes any number (possibly zero) of repetitions of $2$.
\end{itemize}
Define $\cN_1,\cN_2,\cN_3\subset\cN$ by
\begin{align*}
    \cN_1 &= \{\la\in\cN\mid m_1(\la)=0\}, \\
    \cN_2 &= \{\la\in\cN\mid m_i(\la)\le 1 \text{ for } i=1,2,3\}, \\
    \cN_3 &= \left\{\la\in\cN \relmiddle|
        \begin{gathered}
            m_1(\la)=m_3(\la)=0, \ m_2(\la)\le 1, \\
            \text{
                $\la$ does not match $(2k+3,2k,2k-2,\cdots,4,2)$
                for any $k\ge1$
            }
        \end{gathered}
    \right\}.
\end{align*}
    Then 
    \begin{align}
        \cN_1 &\PTsim T^{(14)}_{2,3,4,10,11,12}, \quad
        \cN_2 \PTsim T^{(14)}_{1,4,6,8,10,13}, \quad
        \cN_3 \PTsim T^{(14)}_{2,5,6,8,9,12}.
    \end{align}
    Here $T^{(N)}_{a_1,\cdots,a_k}$
    denotes the set of partitions with parts $\equiv a_1,\dots,a_k \pmod{N}$.
\end{conj}

In the present article
we prove Conjecture \ref{theo:Nandi_conj}.
We also
give corresponding 
$q$-series
identities (like \eqref{eq:RR:q}),
which are missing in Conjecture \ref{theo:Nandi_conj}.
For $a=1,2,3$ we consider a double sum
\begin{equation}\label{eq:asum}
    N_a := 
        \sum_{i,j\ge 0}
        \frac{(-1)^j q^{\binom{i}{2} + 2\binom{j}{2} + 2ij + A_a(i,j)}}
             {(q;q)_i (q^2;q^2)_j},
\end{equation}
where 
$A_1(i,j) = i+j$,
$A_2(i,j) = i+3j$
and
$A_3(i,j) = 2i+3j$.
\begin{theo}\label{theo:asum=prod}
    We have
\begin{alignat}{2}
    \sum_{\la\in\cN_1} q^{|\la|} 
    &= 
    \frac{1}{(q^{2},q^{3},q^{4},q^{10},q^{11},q^{12};q^{14})_\infty}
    &&=
    N_1,
    \label{eq:asum=prod:1} 
    \\
    \sum_{\la\in\cN_2} q^{|\la|} 
    &= 
    \frac{1}{(q^{},q^{4},q^{6},q^{8},q^{10},q^{13};q^{14})_\infty}
    &&=
    N_2,
    \label{eq:asum=prod:2} 
    \\
    \sum_{\la\in\cN_3} q^{|\la|} 
    &= 
    \frac{1}{(q^{2},q^{5},q^{6},q^{8},q^{9},q^{12};q^{14})_\infty}
    &&=
    N_3.
    \label{eq:asum=prod:3}
\end{alignat}
\end{theo}

Obviously, the left equalities in
Theorem \ref{theo:asum=prod} imply Conjecture \ref{theo:Nandi_conj}.

\subsection{Linked partition ideals and regularly linked sets}
\label{sect:intro:outline}

As mentioned above,
a common technique for achieving Step 1 
(in \S\ref{sect:intro:auto_qdiff}) is to use 
\emph{linked partition ideals} (\emph{LPI} for short)
of Andrews
\cite{MR387178},\cite{MR0557013}*{\S 8},
which we review in Appendix \ref{sect:appe:LPI}.
Roughly speaking, it is a subset $\cC\subset\Par$
whose elements can be encoded to infinite sequences (on a certain finite set)
in which certain (finite length) patterns are forbidden to appear.
Theorem \ref{theo:LPI} below is a main result of \cite{MR387178},
and this is applicable for most of known \RRs{} type identities.
\begin{theo}[\cite{MR387178}*{Theorem 4.1}, \cite{MR0557013}*{Theorem 8.11}]
    \label{theo:LPI}
    If $\cC\subset\Par$ is an LPI,
    then one can algorithmically obtain a $q$-difference equation for $f_\cC(x,q)$.
\end{theo}

It is natural to hope to apply this to Nandi's conjectures,
but one can prove that
the set $\cN$ (and $\cN_a$ for $a=1,2,3$) is not an LPI.
Roughly speaking, this is because
while elements of $\cN$ can be encoded to certain infinite sequences 
(Proposition \ref{theo:nandi:lpi:pre}),
there are arbitrarily long forbidden patterns
which originally come from the condition \Na{6}.
Hence the theory of LPIs is not applicable in an obvious way.
However, we can still derive a $q$-difference equation
for $f_{\cN}(x,q)$ (and $f_{\cN_a}(x,q)$ for $a=1,2,3$ as well).
We show this 
in a generalized and algorithmic manner
in \S\ref{sect:automaton},
and apply it to Nandi's conjectures
in \S\ref{sect:nandi:DFA}.

In \S\ref{sect:automaton}
we extend the theory of LPIs using \emph{finite automata};
we consider a class of subsets $\cC\subset\Par$ such that,
roughly speaking,
the elements of $\cC$ can be encoded to infinite sequences (on a certain finite set)
in which certain patterns given by a \emph{regular language} 
(in the sense of \emph{formal language theory}; see Definition \ref{defi:DFA})
are forbidden to appear,
and we say such $\cC$ is \emph{regularly linked} (Definition \ref{defi:regularWPI}).
This notion generalizes LPIs (Proposition \ref{exam:LPI}),
and we show that $\cN$ and $\cN_a$ ($a=1,2,3$)
are regularly linked (Example \ref{exam:nandi:reg}).

\begin{theo}[Theorem \ref{theo:reso:cor} $+$ Appendix \ref{sect:MurrayMiller}] \label{theo:main}
If a subset $\cC\subset\Par$ is regularly linked,
then one can algorithmically obtain a 
$q$-difference equation for $f_{\cC}(x,q)$.
\end{theo}

As an application of the main result above,
in \S\ref{sect:nandi:DFA} we get
a $q$-difference equation for 
$f_{\cN_a}(x,q)$ (for $a=1,2,3$) automatically (Proposition \ref{theo:qd}),
finishing Step 1 for Nandi's conjectures.
We solve these equations in \S\ref{sect:psum},
finishing Step 2. 
The technique used there
seems to be common in dealing with such equations;
indeed, the flow of \S\ref{sect:psum} is similar to 
\cite{MR225741},
\cite{1809.06089}*{Proposition 2.2, Proposition 2.3}, etc.
Finally, Step 3 
is done (also in \S\ref{sect:psum})
by employing three identities of Slater \cite{MR0049225}.

As we see in \S\ref{sect:auto:theo}, 
once Theorem \ref{theo:main} is expressed in terms of finite automata
its key part (Theorem \ref{theo:reso:cor}) is proved immediately from
an almost trivial lemma (Lemma \ref{theo:L(Mv):rec}).
Nevertheless, its application to a concrete problem can be nontrivial
(such as Proposition \ref{theo:qd}) and 
this generalization of LPIs seems worth writing down the details
as it works well in solving Nandi's conjectures.
We hope that the regulary linked sets would be widely used as a method of algorithmic derivation of $q$-difference equations in
the theory 
of partitions 
like the WZ method 
in hypergeometric summations.

\subsection*{Organization of the paper}

In \S\ref{sect:LPI}
we rephrase the defining conditions for $\cN$
as certain forbidden patterns and prefixes
on a certain finite set.
In \S\ref{sect:relang} 
we recall 
standard
definitions and facts
in
formal language theory
(some details are put in Appendix \ref{sect:appe:FLT}).
In \S\ref{sect:regularWPI}
we define regularly linked sets
and in \S\ref{sect:auto:theo}
show Theorem \ref{theo:main}.
In \S\ref{sect:nandi:DFA}
we obtain $q$-difference equations
for $f_{\cN_a}(x,q)$ ($a=1,2,3$)
using 
the results in 
\S\ref{sect:auto:theo}
(we also need
the Modified Murray--Miller Theorem
reviewed in Appendix \ref{sect:MurrayMiller},
which is given in \cite{MR387178}
and constitutes the final step in
Theorem \ref{theo:LPI} (and Theorem \ref{theo:main}).
We apply it explicitly
in Appendix \ref{sect:qd:proof}).
In \S\ref{sect:psum}
we solve these equations,
proving Theorem \ref{theo:asum=prod}.
In Appendix \ref{sect:appe:minFP}
we give supplementary results
regarding Theorem \ref{theo:main}.
In Appendix \ref{sect:appe:LPI}
we review LPIs and compare it to our results.

\section{Nandi's partitions $\cN$}\label{sect:LPI}
\subsection{Multiplicity vectors}
\label{sect:Par}

A partition $\la\in\Par$ can be identified with 
its multiplicity vector $(f_i)_{i\ge 1}$, 
where $f_i=m_i(\la)$.
By this, we have a bijection
\begin{equation}\label{eq:ParPar}
\widehat{\phantom{a}}:\Par\stackrel{\sim}{\longrightarrow}
    \hPar:=\left\{(f_i)_{i\ge1} 
        \in(\bbZ_{\ge 0})^{\bbZ_{\ge1}} \relmiddle| \#\{i\ge1\mid f_i>0\} < \infty \right\},
\end{equation}
and we denote 
the image in $\hatt\Par$ (via this bijection) of 
$\la\in\Par$ and $\cC\subset\Par$, say, 
by $\hatt\la$ and $\hatt\cC$.
It is easy to see
for any $\la\in\Par$
and
$k,d\ge 1$ that
\begin{equation}\label{eq:diffcond}
    1\le \forall i\le \ell(\la)-k,\,\la_i-\la_{i+k}\ge d 
    \iff \forall j\ge 1,\,f_j+\cdots+f_{j+d-1}\le k.
\end{equation}

\begin{lemm}\label{theo:Naa}
The set $\hatt\cN$
consists of $(f_i)_{i\ge 1}$ satisfying
\Naa{1}-\Naa{4}, \Naa{5a}, \Naa{5b} and \Naak{} \textup{(}for all $k\ge0)$.
Here, 
for $i=1,\dots,4,\text{\textup{5a}},\text{\textup{5b}}$,
the condition \Naa{$i$} is given by
\begin{itemize}
    \item[\Naa{$i$}$:$] there are no $j\ge 1$ such that \Pia{$i$}, 
    where 
\begin{itemize}
    \item[\Pia{1}$:$] 
        $(f_j,f_{j+1}) = (\ge 1,\ge 1)$,
    \item[\Pia{2}$:$] 
        $f_j + f_{j+1} + f_{j+2} \ge 3$,
    \item[\Pia{3}$:$] 
        $(f_{j},f_{j+1},f_{j+2},f_{j+3}) = (\ge 1,0,0,\ge 2)$,
    \item[\Pia{4}$:$] 
    \mbox{$(f_{2j},f_{2j+1},f_{2j+2},f_{2j+3}) = (\ge 2,0,0,\ge 1)$},
    \item[\Pia{5a}$:$] 
        $(f_{2j-1},f_{2j},f_{2j+1},f_{2j+2},f_{2j+3}) = \text{\mbox{$(\ge2,0,0,0,\ge1)$}}$,
    \item[\Pia{5b}$:$] 
        $(f_{2j-1},f_{2j},f_{2j+1},f_{2j+2},f_{2j+3}) = \text{\mbox{$(\ge 1,0,0,0,\ge 2)$}}$,
\end{itemize}
\end{itemize}
and the condition \Naak{} $(k\ge0)$ is given by
\begin{itemize}
    \item[\Naak$:$] there are no $j\ge 1$ such that \\
        $(f_{j},f_{j+1},\cdots,f_{j+2k+6}) = (\ge 2,0,0,1,\underbrace{0,1,0,\cdots,1,0,1}_{2k},0,0,\ge 1)$.
\end{itemize}
Here, 
for $n\ge 2$,
we wrote $(x_1,x_2,\dots,x_{n-1},x_n)=(\ge y_1,y_2,\dots,y_{n-1},\ge y_n)$
to mean $x_1\ge y_1$,
$x_i=y_i$ \textup{(}for $2\le i\le n-1$\textup{)} and
$x_n\ge y_n$.
\end{lemm}
\begin{proof}
  It is clear that \Na{1}$\Llra$\Naa{1}.
    That \Na{2}$\Llra$\Naa{2} is a special case of 
    \eqref{eq:diffcond}.
    The condition \Na{3} is equivalent to that 
    $\la$ 
    does not match
    $(j+3,j+3,j)$ for $j\ge 1$,
    which is precisely \Naa{3}.
    Similarly we have
    $\Na{4}\Llra\Naa{4}$ and
    $\Na{5}\Llra\Naa{5a},\Naa{5b}$.
    For \Na{6}, 
    the condition
    $\Naak$
    is equivalent to that
    $(\la_1-\la_2,\dots,\la_{\l(\la)-1}-\la_{\l(\la)})$ does not match 
    $(3,2^k,3,0)$.
\end{proof}

\subsection{Encoding $\cN$ into infinite sequences}
We write
\begin{alignat}{2}
    f_{\le m} &:= (f_1,\dots,f_m,0,0,\dots), &\quad
    \la_{\le m} &:= (\la_{\l'+1},\dots,\la_{\l(\la)})
\end{alignat}
for $m>0$ and $f=(f_i)_{i\ge1}\in\hPar$, $\la\in\Par$, 
where $\l':=\#\{i\ge 1\mid \la_i>m\}$.
It clearly holds $\widehat{\la_{\le m}}=f_{\le m}$ when $\widehat{\lambda}=f$.
Furthermore, for $\cC\subset\Par$ we write
\begin{align*}
    \cC_{\le m}&:=\{\la\in\cC\mid\la=\la_{\le m}\}.
\end{align*}

\begin{defi}
Define two maps $\rsh,\lsh\colon\hPar\lra\hPar$ by
\begin{align}
    \rsh((f_1,f_2,\cdots)) &= (0,f_1,f_2,\cdots), \quad
    \lsh((f_1,f_2,\cdots)) = (f_2,f_3,\cdots).
\end{align}
Abusing the notation, we also regard $\rsh,\lsh$ as
maps from $\Par$ to $\Par$:
\begin{align}
    \rsh((\la_1,\cdots,\la_\ell)) &= (\la_1+1,\cdots,\la_\ell+1), \\
    \lsh((\la_1,\cdots,\la_\ell)) &= (\la_1-1,\cdots,\la_{\ell'}-1),
    \ \text{where $\ell':=\#\{i\ge1\mid \la_i>1\}$.}    
\end{align}
\end{defi}

The following lemma is essentially \cite{MR0557013}*{Lemma 8.9},
but we review the proof
since we assume weaker conditions (on $\cC$).
See Appendix \ref{sect:appe:PI} for
comparison to original arguments in \cite{MR387178},\cite{MR0557013}*{\S8}.
For $\la,\mu\in\Par$,
let $\la\oplus\mu$ be the partition obtained by reordering
$(\la_1,\cdots,\la_{\ell(\la)},\mu_1,\cdots,\mu_{\ell(\mu)})$ in the non-increasing order.
In terms of $\hPar$, it means $(f_i)_{i\ge1} \oplus (g_i)_{i\ge1} = (f_i+g_i)_{i\ge1}$.

\begin{lemm}\label{theo:lpi:block}
    If a subset $\cC\subset\Par$ and 
    an integer $m\in\bbZ_{>0}$
    satisfy
    \begin{equation}\label{eq:mWPImodulus}
        \la\in\cC \implies \la_{\le m}\in\cC
        \qquad\text{and}\qquad
        \lsh^m(\cC) \subset \cC,        
    \end{equation}
    then for each $\la\in\cC$ there uniquely exists a sequence
    $\la^{(1)},\la^{(2)},\cdots$ in $\cC_{\le m}$ such that
    \[
        \la = \la^{(1)} \oplus \rsh^m(\la^{(2)}) \oplus \rsh^{2m}(\la^{(3)}) \oplus \cdots.
    \]
\end{lemm}
\begin{proof}
    Let $f=(f_i)_{i\ge1}:=\hatt\la$.
    Obviously $\hatt{\la^{(i)}}$ must be
    $(f_{1+m(i-1)},\dots,f_{mi},0,0,\dots)$ ($=(\lsh^{m(i-1)}(f))_{\le m}$)
    and hence is unique.
    On the other hand,
    by the assumption
    we see 
    $\lsh^{m(i-1)}(\la)\in\cC$
    and hence
    $(\lsh^{m(i-1)}(\la))_{\le m}\in\cC_{\le m}$.
\end{proof}

\begin{lemm}\label{theo:N':parideal}
    The set $\cN$ 
    satisfies \eqref{eq:mWPImodulus} with $m=2$.
\end{lemm}
\begin{proof}
The conditions 
\Naa{1}-\Naa{3}, \Naak{}
    (resp.\,\Naa{4}, \Naa{5a}, \Naa{5b})
    are stable under 
    $\lsh$ (resp.\,$\lsh^2$)
and the all the conditions 
\Naa{1}-\Naak{} are stable under
$(\hPar\ni)\,f\mapsto f_{\le m}$ for any $m>0$.
\end{proof}

\subsection{Forbidden patterns and prefixes}
To avoid confusion
We denote the empty partition by
$\emptypar\in\Par$.

\begin{defi}\label{defi:pidata}
    \noindent (1)
    For a nonempty set $I$, we write
    \begin{equation}\label{eq:Seq(I)}
        \Sinf(I) := 
        \{(i_1,i_2,\dots)\mid i_j\in I\}
        \ (=I^{\bbZ_{\ge 1}}).
    \end{equation}

    \noindent (2)
    For a triple $(I,m,\pi)$ where
    $I$ is a nonempty set, $m\in\bbZ_{>0}$ and 
    $\pi\colon I\lra\Parr{\le m}$ is a map,
    we define 
    \begin{align}
        \Sinf(I,\pi) &:= 
            \{(i_1,i_2,\dots)\in\Sinf(I) 
            \mid \#\{j\ge1\mid\pi(i_j)\neq\emptypar\}<\infty\} \label{eq:S(I)}
    \end{align}
    and
    $\pi^\pst\colon \Sinf(I,\pi)\lra\Par$ by
    \begin{equation}\label{eq:pi}
    \pi^\pst(i_1,i_2,i_3,\cdots) 
    := \pi(i_1) \oplus \rsh^m(\pi(i_2)) \oplus \rsh^{2m}(\pi(i_3)) \oplus \cdots.
\end{equation}
\end{defi}

Now
$
\cN_{\le2} =\{\la\in\cN\mid \la_1\le 2\} = \{\pi_i \mid i\in I \}
$
where $I=\{0,1,2,3,4\}$ and 
\begin{equation}\label{eq:nandi:pi}
    \pi_0 = \emptypar,  \quad
    \pi_1 = (2),  \quad
    \pi_2 = (2,2),  \quad
    \pi_3 = (1),  \quad
    \pi_4 = (1,1), 
\end{equation}
and $\hatt{\cN}_{\le2}=\{\hatt\pi_i\mid i\in I\}$ is given by
\[
    \hatt\pi_0 = (0,0), \quad
    \hatt\pi_1 = (0,1), \quad
    \hatt\pi_2 = (0,2), \quad
    \hatt\pi_3 = (1,0), \quad
    \hatt\pi_4 = (2,0).    
\]
Here we simply wrote $\hatt\pi_i=(f_1,f_2)$ instead of
$\hatt\pi_i=(f_1,f_2,0,0,\dots)$.
Moreover let us write 
$\pi\colon I\lra\Par_{\le2}\,;\,i\mapsto\pi_i$
(using the same symbol).

By Lemma \ref{theo:lpi:block} and Lemma \ref{theo:N':parideal},
we see $\cN$ (and hence $\cN_a$, $a=1,2,3$) is in bijection with a subset of 
$\Seq(I,\pi)$,
and the condition that
$(i_1,i_2,\dots)\in \Seq(I,\pi)$ is in the image of $\cN$ (resp.\ $\cN_a$) is 
as follows:

\begin{prop}\label{theo:nandi:lpi:pre}
    For $\vi=(i_1,i_2,\cdots)\in \Seq(I,\pi)$,
    it holds
    $\pi^\pst(\vi)\in\cN$ 
    if and only if
    $\vi$ does not match
    any of 
    \begin{equation}\label{eq:nandi:M}
        \begin{gathered}
            (1, \{2,3,4\}),\quad
            (2, \{1,2,3,4\}),\quad
            (3, \{2,4\}),\quad
            (4, \{2,3,4\}),\\
            (1,0,4),\quad
            (2,0,\{3,4\}),\quad
            (3,0,4),\quad
            (4,0,4),\quad
            (4,1^*,0,3).
            \end{gathered}
    \end{equation}
    Here, 
    for $x,y,\dots\in I$,
    $\{x,y,\cdots\}$ means exactly one occurence of one of $x,y,\cdots$,
    and $x^*$ means zero or more repetitions of $x$
    $($see also \eqref{eq:reg_op}$)$.
\end{prop}
\begin{proof}
    It is straightforward to check
    the conditions 
    \Naa{1}-\Naa{5} and \Naak{} ($k\ge0$)
    corresponds to forbidding
    the patterns in Table \ref{table:forbidden} and \ref{table:forbidden2}.
    \begin{table}[h]
        \centering

    \begin{tabular}{ccc}
        & ($j$: odd) & ($j$: even) \\ \hline
        \Naa{1} & & $(1,3), (1,4), (2,3), (2,4)$ \\[1mm]
        \Naa{2} & 
        $\begin{gathered}
            (1,4), (2,3), (2,4) \\[-1mm] (3,4), (4,3), (4,4)
        \end{gathered}$ &
        $\begin{gathered}
            (1,2), (1,4), (2,1), \\[-1mm] (2,2), (2,3), (2,4)
        \end{gathered}$
        \\[3mm]
        \Naa{3} & $(3,2), (4,2)$ & $(1,0,4), (2,0,4)$ \\[.5mm]
        \Naak{} $(k\ge0)$ & $(4,1,1^k,0,3), (4,1,1^k,0,4)$ & $(2,0,3,3^k,1), (2,0,3,3^k,2)$ \\ \hline
    \end{tabular}
    \caption{Forbidden patterns corresponding to 
    \Naa{1}-\Naa{3}, \Naa{6}
    }
    \label{table:forbidden}

    \begin{tabular}{cc}
        \hline
        \Naa{4} & $(2,0,3)$, $(2,0,4)$ \\[.2mm]
        \Naa{5a} & $(4,0,3)$, $(4,0,4)$ \\[.2mm]
        \Naa{5b} & $(3,0,4)$, $(4,0,4)$ \\ \hline
    \end{tabular}
    \captionsetup{width=\linewidth}
    \caption{Forbidden patterns corresponding to 
    \Naa{4}, \Naa{5a}, \Naa{5b}
    }
    \label{table:forbidden2}
    \end{table}

    It is also easy to see that forbidding all the patterns in 
    Table \ref{table:forbidden} and \ref{table:forbidden2} is
    equivalent to forbidding the patterns in \eqref{eq:nandi:M}.
    (For example, $(2,0,3,3^k,1)$ in \Naak{} is redundant since $(2,0,3)$ is in \eqref{eq:nandi:M}.)
\end{proof}

\begin{prop}\label{theo:nandi:M'}
    For $\vi=(i_1,i_2,\cdots)\in \Seq(I,\pi)$ satisfying
    $\pi^\pst(\vi)\in\cN$,
    it holds
    $\pi^\pst(\vi)\in\cN_a$ $(a=1,2,3)$
    if and only if
    $\vi$ does not start with
    any of 
    \begin{equation}\label{eq:nandi:M'}
        \begin{tabular}{cc}
            $(3)$, $(4)$, & $(a=1)$, \\[1mm]
            $(2)$, $(4)$, $(0,4)$, & $(a=2)$, \\[1mm]
            $(2)$, $(3)$, $(4)$, $(0,4)$, $(1^*,0,3)$, & $(a=3)$.
        \end{tabular}
    \end{equation}
\end{prop}

\begin{proof}
For $\cN_1$, 
the additional condition 
$m_1(\la)=0$ 
is equivalent to 
that 
$\vec{i}$
cannot start with
any of 
$(3),(4)$.

For $\cN_2$, 
the additional condition 
$m_i(\la)\le 1$  for $i=1,2,3$ 
is equivalent to 
that 
$\vec{i}$
cannot start with
any of 
$(2),
(4),
(\{0,1,2,3,4\},4)$,
which is equivalent to, after reducing redundancy,
that $\vec{i}$ cannot start with any of 
$(2),
(4),
(0,4)$.    

For $\cN_3$, 
the additional condition that
$m_1(\la)=m_3(\la)=0$, $m_2(\la)\le 1$ and 
$\la$ does not match 
$(2k+3,2k,2k-2,\cdots,4,2)$ (for $k\ge 1$)
is equivalent to 
that 
$\vec{i}$
cannot start with
any of 
$
(2),
(3),
(4),
(\{0,1,2,3,4\},\{3,4\}),
(1^k,0,\{3,4\})
$ (for $k\ge1$).
That is equivalent to, after reducing redundancy,
that $\vec{i}$ cannot start with any of 
$(2),
(3),
(4),
(0,4),
(1^*,0,3)$.
\end{proof}

\section{A formal language theoretic approach}\label{sect:automaton}

In \S\ref{sect:automaton}
we assume $\AL$ is a nonempty finite set.
Let $\AL^*=\bigsqcup_{n\ge 0}\AL^n$ 
be the monoid of words on $\AL$;
elements of $\AL^*$ are written like $i_1\cdots i_n$ (where $i_j\in\AL$),
and the monoid multiplication is concatenation.
Let $\varepsilon$ be the empty word
(i.e.,\,$\AL^0=\{\varepsilon\}$)
and put
$\AL^+:=\AL^*\sm\{\varep\}$.
A \emph{language} (over $\AL$) is a subset of $\AL^*$.
We write
the \emph{product} of 
$X,Y\subset \AL^*$
and the \emph{Kleene star} of $X\subset\AL^*$
as
\begin{equation}\label{eq:reg_op}
XY:=\{ab\mid a\in X,b\in Y\},
\qquad
X^*:=\bigcup_{n\ge 0}X^n.
\end{equation}

\subsection{Regular languages and finite automata}
\label{sect:relang}

\begin{defi}[\cite{sipser13}*{Definition 1.5}]\label{defi:DFA}
A \emph{deterministic finite automaton} (or \emph{DFA} for short)
over $\AL$ is a 5-tuple
$\AUT=(Q,\AL,\delta,s,F)$ where
$Q$ is a finite set (the set of \emph{states}),
$\delta\colon Q\times \AL\lra Q$ (the \emph{transition function}),
$s\in Q$ (the \emph{start state}) and
$F\subset Q$ (the set of \emph{accept states}).
\end{defi} 

\begin{defi}
For a DFA $\AUT=(Q,\AL,\delta,s,F)$,
we define $\hatt\delta\colon Q\times \AL^*\lra Q$
inductively by $\hatt\delta(q, \vep):=q$ and
$\hatt\delta(q, wa):=\delta(\hatt\delta(q,w), a)$ ($q\in Q$, $a\in \AL$, $w\in \AL^*$).
Let $L(\AUT)$ denote the language that $\AUT$ \emph{recognizes} (or \emph{accepts}),
i.e.,\,$L(\AUT)=\{w\in \AL^*\mid\hatt\delta(s,w)\in F\}$.
\end{defi} 

\begin{defi}[\cite{sipser13}*{Definition 1.16}]
A language $X\subset\AL^*$ 
is called \emph{regular}
(or \emph{rational})
if there exists a DFA recognizing $X$.
\end{defi}

\begin{exam}
    The empty set $\emptyset$, singletons $\{i\}$ ($i\in\AL$) 
    and $\AL$ are regular.
\end{exam}

\begin{prop}[See e.g.,\,\cite{sipser13}*{Theorem 1.25, 1.47, 1.49}]
    \label{prop:DFA:1}
    If $Y,Z\subset \AL^*$ are regular,
    so are
    $Y\cap Z$,
    $YZ$,
    $Y^c$
    and
    $Y^*$
    (and thus
    $Y\cup Z$ and
    $Y\sm Z$ as well).
\end{prop}

We review
an algorithmic proof of Proposition \ref{prop:DFA:1}
in Appendix \ref{sect:appe:FLT}.

\vspace{1mm}

For any regular language $X\subset\AL^*$,
there uniquely exists a DFA 
recognizing $X$ 
with the fewest states
(up to isomorphism of DFAs, i.e.,\,renaming of states),
called the \emph{minimal DFA} recognizing $X$.
For a DFA $M$, we denote by $M_{\min}$ 
the minimal DFA such that $L(M_{\min})=L(M)$.
There is an algorithm to compute $M_{\min}$ from a given DFA $M$
(see Remark \ref{prop:DFA:2}).

\begin{coro}\label{theo:DFA_equiv}
    $L(M)=L(N)\Llra M_{\min}\simeq N_{\min}$
    for DFAs $M$ and $N$.
\end{coro}

\subsection{Regularly linked sets}
\label{sect:regularWPI}

Recall Definition \ref{defi:pidata}.

\begin{defi} 
    For
    $S\subset \Seq(\AL)$ 
    and $X,X'\subset \AL^+$
    (recall we assumed $n>0$ in Definition \ref{defi:match} for simplicity),
    we write 
    \[
        \Avoid{S}{X}{X'} := 
        \left\{\vec{i}\in S\relmiddle|
            \begin{aligned}
                &\text{$\forall\vec{j}\in X$, $\vi$ does not match $\vec{j}$, and} \\
                &\text{$\forall\vec{j}\in X'$, $\vi$ does not begin with $\vec{j}$}
            \end{aligned}
        \right\}.
    \]
    In other words,
    \begin{equation}\label{eq:avoid}
        \Avoid{S}{X}{X'} 
        = S\cap\comp{\big(\AL^*\sm (\AL^*X\AL^*\cup X'\AL^*)\big)}
    \end{equation}
    where 
    we write
    for $A\subset\AL^*$
    \[
        \comp{A} := 
        \{(i_j)_{j\ge1}\in\Seq(\AL)\mid
        \forall n\ge1,\,i_1\cdots i_n\in A\}.
    \]
\end{defi}

\begin{defi}\label{defi:regularWPI}
    We say that a subset 
    $\cC\subset\Par$ is \emph{regularly linked}
    if 
    there exists a 5-tuple $(m,I,\pi,X,X')$
    consisting of an integer 
    $m>0$,
    a nonempty finite set $I$,
    an injective map $\pi\colon I\lra \Par_{\le m}$ 
    (hence $\pi^\pst\colon\Seq(I,\pi)\lra\Par$
    is injective)
    and
    regular languages
    $X,X'\subset I^+$
    such that
    \begin{equation}\label{eq:regularWPI}
        \pi^\pst(\Avoid{\Seq(I,\pi)}{X}{X'}) = \cC.
    \end{equation}
\end{defi}

\begin{exam}\label{exam:nandi:reg}
The sets $\cN$ and $\cN_a$ ($a=1,2,3$)
(see Conjecture \ref{theo:Nandi_conj}) 
are regularly linked.
Indeed,
Proposition \ref{theo:nandi:lpi:pre} and \ref{theo:nandi:M'}
translate as
\begin{gather}
    \pi^\pst(\Avoid{\Sinf(I,\pi)}{\FB_\cN}{\emptyset}) = \cN,
    \\
    \pi^\pst(\Avoid{\Sinf(I,\pi)}{\FB_\cN}{\FB_{\cN_a}}) = \cN_a
    \label{eq:nandi:pi:Na}
\end{gather}
where $I=\{0,1,2,3,4\}$,
$\pi\colon I\lra\Par_{\le 2}\,;\,i\mapsto\pi_i$
is given by \eqref{eq:nandi:pi},
and 
\begin{gather}
    X_\cN=
    \left\{\begin{gathered}
        12,13,14,
        21,22,23,24,
        32,34,
        42,43,44,\\
        104,203,204,304,404
    \end{gathered}\right\}\cup\{4\}\{1\}^*\{03\},
    \quad
    \label{eq:nandi:M:lang}
    \\
    \FB_{\cN_1} = \{3,4\}, \,
    \FB_{\cN_2} = \{2,4,04\}, \,
    \FB_{\cN_3} = \{2,3,4,04\}\cup\{1\}^*\{03\} 
    \quad
    \label{eq:nandi:X':lang}
\end{gather}
are the regular languages over $I$
consisting of the patterns in \eqref{eq:nandi:M}
and \eqref{eq:nandi:M'}.
\end{exam}

\begin{rema}
In Definition \ref{defi:regularWPI}, $X$ is superfluous
since we can write
\[
\Avoid{\Seq(I,\pi)}{X}{X'}=\Avoid{\Seq(I,\pi)}{\emptyset}{X'\cup \AL^*X}.
\]
Note that
if $X(\subseteq \AL^+)$ is regular then so is $\AL^*X(\subseteq \AL^+)$
(see Proposition \ref{prop:DFA:1}). 
Nevertheless, it seems more consistent with human intuition
to separate some forbidden patterns from forbidden prefixes
as seen in Example \ref{exam:nandi:reg} (see also Proposition \ref{exam:LPI}).
\end{rema}

\subsection{The main construction}
\label{sect:auto:theo}

\begin{defi}\label{eq:Mv}
For a DFA $M=(Q,\AL,\delta,s,F)$ and $v\in Q$
we write $M_v:=(Q,\AL,\delta,v,F)$, that is, the same DFA as $M$ except that its start state is $v$.
\end{defi}

\begin{defi}\label{defi:j_vi}
    For a nonempty set $I$, 
    we write
    \begin{alignat}{2}
        j\cdot\vec{i} &:= (j,i_1,i_2,\dots) \in\Seq(I)
        &\quad&\text{for $j\in I$, $\vec{i}=(i_1,i_2,\dots)\in\Seq(I)$},
        \\
        j\cdot S &:= \{j\cdot\vec{i}\mid\vec{i}\in S\} \subset\Seq(I)
        &&\text{for $j\in I$, $S\subset \Seq(I)$}.
    \end{alignat}
\end{defi}

\begin{lemm}\label{theo:L(Mv):rec}
    Let $M=(Q,I,\de,s,F)$ be a DFA.
    For $v\in Q$ we have
    \[
        \comp{(L(M_v)^c)} = \hspace{-2mm}\bigsqcup_{\substack{a\in I \\ \de(v,a)\notin F}} \hspace{-1mm} a\cdot \comp{(L(M_{\de(v,a)})^c)}.
    \]
\end{lemm}
\begin{proof}
    By
    $\comp{(L(M_v)^c)}=\{(a_i)_{i\ge1}\in\Seq(I)\mid 
        \forall n\ge1,\,\hatt\de(v,a_1\cdots a_n)\notin F
    \}$.
\end{proof}

For $\la\in\Par$ we write
$\text{wt}(\la) := x^{\len(\la)}q^{|\la|}$.
Assume a map $\pi\colon I\lra\Par_{\le m}$ is given.
For $\vec{i}\in\Sinf(I,\pi)$ and $j\in I$, we have
\begin{equation}\label{eq:wt_ji}
    \twt(\pi^\pst(j\cdot\vec{i})) 
    = \twt(\pi(j))\cdot \big(\twt(\pi^\pst(\vec{i}))|_{x\mapsto xq^m}\big)
\end{equation}
by $\pi^\pst(j\cdot\vec{i}) = \pi(j)\oplus\rsh^m(\pi^\pst(\vec{i}))$ (and
$\twt(\rsh(\la))=\twt(\la)|_{x\mapsto xq}$).

\begin{theo}\label{theo:reso:cor}
    Assume $\cC\subset\Par$ is regularly linked
    and let $m,I,\pi,X,X'$ be as in Definition \ref{defi:regularWPI}.
    Let $M=(Q,I,\de,s,F)$ be a DFA recognizing $I^*XI^*\cup X'I^*$.
    Note that $s\in Q\sm F$ since $\varep\notin I^*XI^*\cup X'I^*=L(M)$.
    Define
    \begin{equation}\label{eq:reso:cor:cC}
        \cC^{(v)} :=
        \pi^\pst\big(\Seq(I,\pi)\cap\comp{(L(M_v)^c)}\big)
    \end{equation}
    for
    $v\in Q\sm F$.
    Then $\cC^{(s)}=\cC$ and we have a system of $q$-difference equations
    \begin{equation}\label{eq:reso:sysqdiff}
        f_{\cC^{(v)}}(x,q)
        = \sum_{u\in Q\sm F} 
        \bigg(\hspace{-1mm}\sums{a\in I \\ u=\de(v,a)} \hspace{-2mm} x^{\l(\pi_{a})} q^{|\pi_{a}|}\bigg)
        f_{\cC^{(u)}}(xq^m,q)
        \quad (v\in Q\sm F).
    \end{equation}
\end{theo}

\begin{proof}
The fact $\cC^{(s)}=\cC$ is obvious by $M_s=M$ and \eqref{eq:avoid}.
    Put
    $S_v:=\Seq(I,\pi)\cap\comp{(L(M_v)^c)}$.
    Then by Lemma \ref{theo:L(Mv):rec} we have 
\[
S_v = \hspace{-2mm}\bigsqcup_{\substack{a\in I \\ \de(v,a)\notin F}} \hspace{-1mm} a\cdot S_{\de(v,a)}.
\]
    Apply the map
    $\Seq(I,\pi)\supset S\mapsto \sum_{\vec{i}\in S} \twt(\pi^\pst(\vec{i}))$.
    Since
    $\pi^\pst$ is injective and
    $f_\cC(x,q) = \sum_{\la\in\cC}\twt(\la)$ for $\cC\subset\Par$ (see \eqref{eq:defi:f_C}),
    $S_v$ is then mapped to $f_{\cC^{(v)}}(x,q)$.
    Hence by 
    \eqref{eq:wt_ji}
    we get \eqref{eq:reso:sysqdiff}.
\end{proof}

\begin{rema}\label{rema:reso:Xv}
    In Theorem \ref{theo:reso:cor},
    we can explicitly determine $\cC^{(v)}$
    if $v\in Q\sm F$ is reachable, i.e.,\,$v=\hatt\delta(s,w)$ for some 
$w\in\AL^*$ (for example, every state in a minimal DFA is reachable).
In Appendix \ref{sect:appe:minFP} we show
    \[
        L(M_v) = I^*XI^*\cup X''I^*
        \quad\text{for some $X''\subset I^+$},
    \]
    and 
    explicitly find
    the minimum such $X''$,
    namely,
    the regular language 
    $X_v$ given in \eqref{eq:Xv} (with $\AL:=I$).
    Hence
    for $v\in Q\sm F$ we have by \eqref{eq:avoid}
    \begin{equation}\label{eq:L(Mv)c:comp}
        \comp{(L(M_v)^c)} = \Avoid{\Seq(I)}{\FB}{\FB_v}.
    \end{equation}
Thus, $\cC^{(v)}$ is regularly linked with forbidden patterns $X$ and prefixes $X_v$:
    \[
        \cC^{(v)}
        =
        \pi^\pst(\Avoid{\Sinf(I,\pi)}{\FB}{\FB_v})
        \quad(\subset \cC).
    \]
\end{rema}

\begin{proof}[Proof of Theorem \ref{theo:main}]
Since $|Q\sm F|$ is finite in Theorem \ref{theo:reso:cor},
from the system \eqref{eq:reso:sysqdiff}
we can deduce for any $v\in Q\sm F$ 
a single $q$-difference equation for 
$f_{\cC^{(v)}}(x,q)$
by the algorithm given in 
\cite{MR387178}*{p.1040}
(called \emph{Modified Murray--Miller Theorem} therein),
which we review in Appendix \ref{sect:MurrayMiller}.
\end{proof}

\section{A proof of Nandi's conjectures}\label{sect:nandi:proof}

\subsection{Algorithmic derivation of $q$-difference equations}
\label{sect:nandi:DFA}

We apply Theorem \ref{theo:reso:cor} for $\cN$
(recall Example \ref{exam:nandi:reg}).
The resulting system \eqref{eq:reso:sysqdiff}
depends on the choice of a DFA $M$ in Theorem \ref{theo:reso:cor},
and
in this case we can achieve the proof
by taking $M$ minimal.

Let $I=\{0,1,2,3,4\}$ and
let $X=X_{\cN}\subset I^+$ be the regular language given in \eqref{eq:nandi:M:lang}.
Since the proof of Proposition 
\ref{prop:DFA:1}
and Remark \ref{prop:DFA:2}
are constructive,
we can algorithmically find 
(see Remark \ref{rema:GAP})
the minimal DFA that recognizes $I^*XI^*$ is $M=(Q,I,\delta,s,F)$ where
$Q=\{q_0,\dots,q_7\}$, $s=q_0$, $F=\{q_6\}$
and $\delta\colon Q\times I\lra Q$ is given by Table \ref{table:DFA:A},
in which we display
$\delta'(v,j)$ 
such that $\delta(q_v,j)=q_{\delta'(v,j)}$
($v\in \{0,\dots,7\}$, $j\in I$).
See also Figure \ref{fig:DFA:A} (but we will not refer it).

{
\newcommand{\tabhaba}{.25\textwidth}
\newcommand{\semaisp}{\hspace{2mm}}
\newcommand{\semaitab}{@{}c@{\hspace{1mm}}|@{\hspace{1mm}}c@{\semaisp}c@{\semaisp}c@{\semaisp}c@{\semaisp}c}
\begin{figure}[h]
    \begin{minipage}{\tabhaba}
        \centering
        {\small
        \begin{tabular}{\semaitab}
            $v\backslash j$ & 0 & 1 & 2 & 3 & 4 \\ \hline
            0 & 0 & 1 & 2 & 3 & 4 \\
            1 & 5 & 1 & 6 & 6 & 6 \\
            2 & 7 & 6 & 6 & 6 & 6 \\
            3 & 5 & 1 & 6 & 3 & 6 \\
            4 & 7 & 4 & 6 & 6 & 6 \\
            5 & 0 & 1 & 2 & 3 & 6 \\
            6 & 6 & 6 & 6 & 6 & 6 \\
            7 & 0 & 1 & 2 & 6 & 6 
        \end{tabular}
        \captionsetup{type=table} 
        \captionsetup{width=\linewidth}
        \caption{$\delta'(v,j)$}
        \label{table:DFA:A}
        }
    \end{minipage}
    \begin{minipage}{.5\textwidth}
        \centering
\begin{tikzpicture}[scale=.70,->,>=latex',shorten >=1pt,auto,state/.style={circle,draw,scale=.9,inner sep=0pt,minimum size=1.5em}]
    \node[state,initial] at (0,0) (q0) {$q_0$};
    \node[state] (q1) at (2,-1.3) {$q_1$};
    \node[state] (q2) at (4,0) {$q_2$};
    \node[state] (q3) at (3,-3) {$q_3$};
    \node[state] (q4) at (2,2.5) {$q_4$};
    \node[state] (q5) at (4.5,-1.5) {$q_5$};
    \node[state,accepting] (q6) at (7,0) {$q_6$};
    \node[state] (q7) at (3,1.5) {$q_7$};

    \draw (q0) edge[loop above] node[scale=.70] {0} (q0);
    \draw (q0) edge node[scale=.70,left] {1} (q1);
    \draw (q0) edge node[scale=.70] {2} (q2);
    \draw (q0) edge[bend right] node[scale=.70,below left] {3} (q3);
    \draw (q0) edge node[scale=.70] {4} (q4);

    \draw (q1) edge node[scale=.70,below] {0} (q5);
    \draw (q1) edge[loop left] node[scale=.70] {1} (q1);
    \draw (q1) edge node[scale=.70,near end,below] {2,3,4} (q6);

    \draw (q2) edge node[right,scale=.70] {0} (q7);
    \draw (q2) edge node[scale=.70,near start] {1,2,3,4} (q6);

    \draw (q3) edge[bend right] node[scale=.70,right] {0} (q5);
    \draw (q3) edge node[scale=.70,left] {1} (q1);
    \draw (q3) edge[bend right] node[scale=.70,below] {2,4} (q6);
    \draw (q3) edge[loop below] node[scale=.70] {3} (q3);

    \draw (q4) edge node[scale=.70,above] {0} (q7);
    \draw (q4) edge[loop left] node[scale=.70] {1} (q4);
    \draw (q4) edge[bend left] node[scale=.70,above] {2,3,4} (q6);

    \draw (q5) edge node[scale=.70,right,near end] {0} (q0);
    \draw (q5) edge[bend left] node[scale=.70] {1} (q1);
    \draw (q5) edge node[scale=.70,near start,right] {2} (q2);
    \draw (q5) edge node[scale=.70,below] {3} (q3);
    \draw (q5) edge[bend right] node[scale=.70,near start] {4} (q6);

    \draw (q6) edge[loop above] node[scale=.70] {0,1,2,3,4} (q6);

    \draw (q7) edge node[scale=.70,above] {0} (q0);
    \draw (q7) edge node[scale=.70,near start,left] {1} (q1);
    \draw (q7) edge[bend right] node[scale=.70,left] {2} (q2);
    \draw (q7) edge node[scale=.70] {3,4} (q6);
\end{tikzpicture}
        \caption{}
        \label{fig:DFA:A}
\end{minipage}
\end{figure}
}

Writing $F_i(x):=f_{\cN^{(q_i)}}(x,q)$
for $q_i\in Q\sm F$ (i.e.,\,$i\in\{0,\dots,5,7\}$),
by Theorem \ref{theo:reso:cor}
we obtain a system of $q$-difference equations

{\small
\begin{equation}\label{eq:nandi:sysqdiff}
    \begin{pmatrix}
        F_{0}(x) \\ 
        F_{1}(x) \\ 
        F_{2}(x) \\ 
        F_{3}(x) \\ 
        F_{4}(x) \\ 
        F_{5}(x) \\ 
        F_{7}(x)
    \end{pmatrix}
    = 
    \left(\begin{array}{rrrrrrrr}
        1 & x q^{2} & x^{2} q^{4} & x q & x^{2} q^{2} & 0 & 0 \\
        0 & x q^{2} & 0 & 0 & 0 & 1 & 0 \\
        0 & 0 & 0 & 0 & 0 & 0 & 1 \\
        0 & x q^{2} & 0 & x q & 0 & 1 & 0 \\
        0 & 0 & 0 & 0 & x q^{2} & 0 & 1 \\
        1 & x q^{2} & x^{2} q^{4} & x q & 0 & 0 & 0 \\
        1 & x q^{2} & x^{2} q^{4} & 0 & 0 & 0 & 0
    \end{array}\right)
    \begin{pmatrix}
        F_{0}(xq^2) \\ 
        F_{1}(xq^2) \\ 
        F_{2}(xq^2) \\ 
        F_{3}(xq^2) \\ 
        F_{4}(xq^2) \\ 
        F_{5}(xq^2) \\ 
        F_{7}(xq^2)
    \end{pmatrix}.
\end{equation}
}

Moreover, 
it can be algorithmically proved 
(see Remark \ref{rema:GAP})
that
\begin{equation}\label{eq:L(M_qi)}
\begin{aligned}
    L(M_{q_7}) &= I^*XI^*\cup X_{\cN_1}I^*, \\
    L(M_{q_3}) &= I^*XI^*\cup X_{\cN_2}I^*, \\
    L(M_{q_4}) &= I^*XI^*\cup X_{\cN_3}I^*,
\end{aligned}
\end{equation}
where
$X_{\cN_1},X_{\cN_2},X_{\cN_3}\subset I^+$
are as in \eqref{eq:nandi:X':lang};
one can construct
DFAs recognizing the right-hand sides
via Proposition \ref{prop:DFA:1}
and then use Corollary \ref{theo:DFA_equiv}.
Now
by 
\eqref{eq:avoid},
\eqref{eq:nandi:pi:Na}, 
\eqref{eq:reso:cor:cC} and
\eqref{eq:L(M_qi)}
we have
\begin{equation}\label{eq:cN_qi}
\cN^{(q_7)}=\cN_1,\quad
\cN^{(q_3)}=\cN_2,\quad
\cN^{(q_4)}=\cN_3.
\end{equation}

\begin{rema}\label{rema:X_N}
    Alternatively,
    one can show \eqref{eq:cN_qi}
    by computing DFAs 
    recognizing
    $X_{\cN_a}$ ($a=1,2,3$) and
    $X_{q_i}$ (given by \eqref{eq:Xv}; 
    see also Remark \ref{rema:reso:Xv})
    for $q_i\in Q\sm F$
    and check
    $X_{q_7}=X_{\cN_1}$,
    $X_{q_3}=X_{\cN_2}$,
    $X_{q_4}=X_{\cN_3}$
    by Corollary \ref{theo:DFA_equiv}.
\end{rema}

Hence,
we can apply the algorithm described in Appendix \ref{sect:MurrayMiller}
to obtain $q$-difference equations 
for
$f_{\cN_1}(x,q) = F_{7}(x)$,
$f_{\cN_2}(x,q) = F_{3}(x)$ and
$f_{\cN_3}(x,q) = F_{4}(x)$
(the explicit calculation is given in 
Appendix \ref{sect:qd:proof}):

\begin{prop}\label{theo:qd}
For $a=1,2,3$,
    the series
    $f_{\cN_a}(x,q)$ satisfies the $q$-difference equation
    \begin{equation}\label{eq:qd:F}
        0 = \sum_{i=0}^{5} p^{(a)}_{2i}(x,q) f_{\cN_a}(xq^{2i},q),
    \end{equation}
    where $p^{(a)}_{2i}=p^{(a)}_{2i}(x,q)$ are given in
    the following table.
    {
\footnotesize
\[
\begin{tabular}{c@{\hspace{0.5mm}}|@{\hspace{1mm}}l@{\hspace{1mm}}l@{\hspace{1mm}}l}
    & $a=1$ & $a=2$ & $a=3$ \\ \hline
    $p^{(a)}_{0}$ & $1$ & $1$ & $1$
    \\
    $p^{(a)}_{2}$ & 
    $-1 - x(q^{2} + q^{3} + q^{4})$ & 
    $-1 - x(q + q^2 + q^{4})$ &
    $-1 - x(q^{2} + q^{4} + q^{5})$
    \\
    $p^{(a)}_{4}$ & 
    $xq^{4}(1 - x + xq^3 + xq^4 + xq^5)$ & 
    $x q^4 (1 + xq + xq^3)$ &
    $x q^4  (1 + xq^5 + xq^7)$ 
    \\
    $p^{(a)}_{6}$ & 
    $x^{2}q^{6}(-1 + xq^{4}(1 + q + q^{2} - q^{5}))$ & 
    $x^2 q^{10} (-1 + xq^4 + xq^6)$ &
    $x^2 q^{10} (-1 + xq^4 + xq^6)$ 
    \\
    $p^{(a)}_{8}$ & 
    $x^{3}q^{13} (1 + q + q^{2})(1 - xq^{6})$ & 
    $x^3 q^{15} (1 + q^2 + q^3) (1 - xq^6)$ &
    $x^3 q^{18} (1 + q + q^3) (1 - xq^6)$ 
    \\
    $p^{(a)}_{10}$ & 
    $x^{3}q^{17} (1 - xq^{6}) (1 - xq^{8})$ &
    $x^3 q^{19} (1 - xq^6) (1 - xq^8)$ &
    $x^3 q^{23} (1 - xq^6) (1 - xq^8)$ 
\end{tabular}
\]
}
\end{prop}

\vspace{1mm}
\begin{rema}
\label{rema:GAP}
We can use computer algebras in these constructions.
For example,
using a GAP package Automata \cites{GAP4,GAP_Automata1.14}
we can compute $M$ (up to renaming of states) as follows:
{\footnotesize
\begin{verbatim}
 gap> LoadPackage("automata");
 gap> Xn:=RationalExpression("12U13U14U21U22U23U24U32U34U42U43U44U104U203U204U304U404U41*03","01234");
 gap> Is:=RationalExpression("(0U1U2U3U4)*","01234");
 gap> r:=ProductRatExp(Is,ProductRatExp(Xn,Is));
 gap> M:=RatExpToAut(r);
 gap> Display(M);
\end{verbatim}
}

\noindent
We can also check \eqref{eq:L(M_qi)}:
for $\cN_1$, ($\cN_2$ and $\cN_3$ are similar)
{\footnotesize
\begin{verbatim}
 gap> Xn1:=RationalExpression("3U4","01234");
 gap> r1:=UnionRatExp(r,ProductRatExp(Xn1,Is));
 gap> SetInitialStatesOfAutomaton(M,5);
 gap> AreEquivAut(M,RatExpToAut(r1));
\end{verbatim}
}
\noindent
Here, the state 5 (in the third line) corresponds to $q_7$
in our notation.
\end{rema}

\subsection{Solving the equation (\ref{eq:qd:F})}\label{sect:psum}

Recall the \emph{Euler's identities} \cite{MR2128719}*{(II.1),(II.2)}
\begin{equation}
    \sum_{n\ge 0} \frac{x^n}{(q;q)_n} \overset{\refEuler}{=} \frac{1}{(x;q)_\infty}, \qquad
    \sum_{n\ge 0} \frac{q^{\binom{n}{2}} x^n}{(q;q)_n} \overset{\refEulerr}{=} (-x;q)_\infty.
\end{equation}

The following lemma is a formal series version of
Appell's comparison theorem \cite{MR0089895}*{page 101}.

\begin{lemm}\label{theo:Appell}
    For formal series
    $A(x)=\sum_{m\ge0} a_m x^m, B(x)=A(x)/(1-x)=\sum_{n\ge0}b_n x^n$,
    if 
    $\lim_{n\rightarrow\infty} b_n$ exists
    then 
    $(A(1)=)\,\sum_{m\ge 0} a_m = \lim_{n\rightarrow\infty} b_n$.
\end{lemm}
\begin{proof}
    By $B(x)=A(x)/(1-x)$
    we have
    $\sum_{n=0}^{M}a_n = b_M$,
    and let $M\rightarrow\infty$.
\end{proof}

Within the proof below
we freely use the fact 
the $q$-difference equation
$\sum_{i,j,k} a_{ijk} x^iq^jF(xq^k)=0$ 
for a formal series $F(x)=\sum_{M\in\bbZ}f_M x^M$ 
is equivalent to
the recurrence
$\sum_{i,j,k} a_{ijk} q^{k(M-i)+j} f_{M-i}=0$ for all $M\in\bbZ$.

\begin{proof}[Proof of Theorem \ref{theo:asum=prod}]
    We simply write $F_a(x)=f_{\cN_a}(x,q)$.
    First we consider the case $a=1$.
    Define $G_1(x)$ and $\{g^{(1)}_M\}_{M\in\bbZ}$ by
    \begin{equation}\label{eq:qd:F-G}
        G_1(x) 
        = \sum_{M\in\bbZ} g^{(1)}_M x^M
        := \frac{F_1(x)}{(x;q^2)_\infty}.
    \end{equation}
    Note that $g^{(1)}_M=0$ if $M<0$.
    Dividing \eqref{eq:qd:F} by $(xq^6;q^2)_\infty$ yields
    \begin{align*} 
        0=&\ (1-x) (1-x q^{2}) (1-x q^{4}) G_1(x) \\
        &- (1- x q^{2}) (1-x q^{4}) (1 + x q^{2} + x q^{3} + x q^{4}) G_1(xq^{2})\\ 
        &+ x q^{4} (1-x q^{4}) ( 1 - x + x q^{3}  + x q^{4} + x q^{5}) G_1(xq^{4})\\
        &- x^{2} q^{6} (1 - x q^{4} -  x q^{5} -  x q^{6} + x q^{9}) G_1(xq^{6})\\ 
        &+ x^{3} q^{13} (1 + q + q^{2}) G_1(xq^{8})
         + x^{3} q^{17} G_1(xq^{10}),
    \end{align*}
    which is equivalent to
    \begin{equation}\label{eq:qd:g}
    \begin{aligned}
        0=&\ ( 1 - q^{2M} )g^{(1)}_{M} 
        +(-1-q^{2}-q^{4}-q^{2M+1}+q^{4M})^{} g^{(1)}_{M-1} \\
        &+ q^{2} (1+q^{2}+q^{4}-q^{2M-3}) (1+q^{2M-3}) (1+q^{2M-2}) g^{(1)}_{M-2} \\
        &- q^{6}(1-q^{2M-5})^{}(1+q^{2M-5})^{}(1+q^{2M-4})^{}(1+q^{2M-3})^{}(1+q^{2M-2}) g^{(1)}_{M-3}
    \end{aligned}
    \end{equation}
    for all $M\in\bbZ$.
    Letting
    \begin{equation}\label{eq:qd:G-H}
        h^{(1)}_M := \frac{g^{(1)}_M}{(-q;q)_{2M}}
        \quad\text{and}\quad 
        H_1(x) 
        := \sum_{M\in\bbZ}h^{(1)}_M x^M
    \end{equation}
    (note that $h^{(1)}_M=0$ if $M<0$)
    and dividing \eqref{eq:qd:g} by $q^{-1}(-q;q)_{2M-2}$, 
    we have
    \begin{align*}
        0 &= q (1-q^{2M}) (1+q^{2M-1}) (1+q^{2M}) h^{(1)}_{M}
        +q (-1-q^{2}-q^{4}-q^{2M+1}+q^{4M}) h^{(1)}_{M-1}\\
        &\phantom{=}+q^{3} (1+q^{2}+q^{4}-q^{2M-3}) h^{(1)}_{M-2}
        -q^{7} (1-q^{2M-5}) h^{(1)}_{M-3}
    \end{align*}
    for all $M\in\bbZ$,
    which is equivalent to
    \begin{equation}\label{eq:qd:H}
        \begin{aligned} 
            0 &= q (1-x) (1-x q^{2}) (1-x q^{4}) H_1(x) \\
            &\phantom{=}+ (1-x q^{2}) (1 - x q^{4}) (1 + x q^{2}) H_1(xq^{2})
            - q (1-x q^{4}) H_1(xq^{4})
            - H_1(xq^{6}).
        \end{aligned}
    \end{equation}
    Finally we define $I_1(x)$ and $i^{(1)}_M$ ($M\in\bbZ$) by
    \begin{equation}\label{eq:qd:H-I}
        I_1(x) = \sum_{M\in\bbZ}i^{(1)}_M x^M := H_1(x) (x;q^2)_\infty
    \end{equation}
    (note that $i^{(1)}_M=0$ if $M<0$)
    and multiply \eqref{eq:qd:H} by $(xq^6;q^2)_\infty$ to obtain
    \begin{align*} 
        0 &= q I_1(x)
        +(1+xq^{2})I_1(xq^{2})
        -qI_1(xq^{4})
        -I_1(xq^{6})
    \end{align*}
    which is equivalent to
    \[
        0 = q (1-q^{2M}) (1+q^{2M}) (1+q^{2M-1}) i^{(1)}_{M}
        + q^{2M} i^{(1)}_{M-1}
    \]
    for all $M\in\bbZ$.
    Since
    $i^{(1)}_0 = h^{(1)}_0 = g^{(1)}_0 = F_1(0) = f_{\cN_1}(0,q)= 1$,
    we have
    \begin{equation}
        i^{(1)}_M = \frac{(-1)^M q^{M^2}}{(-q;q)_{2M} (q^2;q^2)_M},
        \ \text{i.e.,}\
        I_1(x) = \sum_{M\ge 0} \frac{(-1)^M q^{M^2}}{(-q;q)_{2M} (q^2;q^2)_M} x^M.
    \end{equation}

    The cases $a=2,3$ can be treated parallelly:
    defining
    $G_a(x) = \sum_M g^{(a)}_M x^M$,
    $H_a(x) = \sum_M h^{(a)}_M x^M$,
    $I_a(x) = \sum_M i^{(a)}_M x^M$
    by
    transformations shown below,
{
\small
\[
\begin{tabular}{c|cc}
    $a=1$ & $a=2$ & $a=3$  \\ \hline 
     & & \\[-2mm]
    \eqref{eq:qd:F-G} & 
    $\DS G_2(x) = F_2(x)/(x;q^2)_\infty$ & 
    $\DS G_3(x) = F_3(x)/(x;q^2)_\infty$ 
    \\[3mm]
    \eqref{eq:qd:G-H} & 
    $\DS h^{(2)}_M = g^{(2)}_M / (-q;q)_{2M}$ & 
    $\DS h^{(3)}_M = g^{(3)}_M / (-q^2;q)_{2M}$ 
    \\[3mm]
    \eqref{eq:qd:H-I} & 
    $I_2(x) = H_2(x) (x;q^2)_\infty$ &
    $I_3(x) = H_3(x) (x;q^2)_\infty$
\end{tabular}
\]
}

\noindent
we can get
\begin{equation}\label{eq:i_M}
    I_a(x) = \sum_{M\ge 0} \frac{(-1)^M q^{M(M+2t)}}{(-q^{1+s};q)_{2M} (q^2;q^2)_M} x^M,
\end{equation}
where $(s,t):=(0,0), (0,1), (1,1)$ for $a=1,2,3$ respectively.

For each $a=1,2,3$,
    by \refEuler{} we have
    \[
        H_a(x) 
        = \frac{I_a(x)}{(x;q^2)_\infty} 
        = \sum_{N\ge 0} \frac{x^N}{(q^2;q^2)_N} 
        \sum_{M\ge 0} \frac{(-1)^M q^{M(M+2t)}}{(-q^{1+s};q)_{2M} (q^2;q^2)_M} x^M.
    \]
    Hence by \eqref{eq:qd:G-H}
    \begin{equation}\label{eq:g_M}
        g^{(a)}_L
        = \sum_{0\le M \le L} 
            \frac{(-1)^M q^{M(M+2t)} (-q^{1+s};q)_{2L}}{(q^2;q^2)_{L-M} (-q^{1+s};q)_{2M} (q^2;q^2)_M}
    \end{equation}
    for $L\ge 0$,
    which implies
    \[
        \lim_{L\rightarrow\infty} g^{(a)}_L
        =
        \frac{ (-q;q)_{\infty}}{(q^2;q^2)_{\infty}}
        \sum_{M \ge 0} 
        \frac{(-1)^M q^{M(M+2t)}}{ (-q;q)_{2M+s} (q^2;q^2)_M}. 
    \]
    Since $F_a(x) = (x;q^2)_\infty G_a(x) = (1-x) (xq^2;q^2)_\infty G_a(x)$,
    by Lemma \ref{theo:Appell}
    \begin{align}
        F_a(1) = (q^2;q^2)_\infty \lim_{L\rightarrow\infty}g^{(a)}_L
        &= 
            (-q;q)_{\infty}
            \sum_{M \ge 0} 
            \frac{(-1)^M q^{M(M+2t)}}{(-q;q)_{2M+s} (q^2;q^2)_M}.
            \label{eq:F(1):1sum}
    \end{align}

    Now the left equalities in Theorem \ref{theo:asum=prod}
    follows from 
    three identities 
    due to Slater (\cite{MR0049225}*{(117),(118),(119)}=\cite{MR3752624}*{(A.187),(A.186),(A.188)}
    with $q\mapsto -q$):
    \begin{equation}\label{eq:Slater:mod28}
        \sum_{n\ge 0} \frac{(-1)^n q^{n(n+2t)}}{(-q;q)_{2n+s} (q^2;q^2)_{n}}
        =
        \frac{(q;q^2)_\infty}{(q^2;q^2)_\infty}
        \frac{(q^{2b},q^{14-2b},q^{14};q^{14})_\infty}{(q^{b},q^{14-b};q^{14})_\infty},
    \end{equation}
    where $(b,s,t)=(3,0,0), (1,0,1), (5,1,1)$.

    Also,
    using \refEulerr{} for
    $(-q;q)_{\infty}/(-q;q)_{2M+s}
    = (-q^{2M+1+s};q)_{\infty}$,
    we have
    \[
        \eqref{eq:F(1):1sum}
        = 
        \sum_{M\ge 0} 
        \frac{(-1)^M q^{M(M+2t)}}{(q^2;q^2)_M}
        \sum_{K\ge 0}
        \frac{q^{\binom{K}{2} + (2M+1+s)K}}{(q;q)_K} 
        =
        N_a,
    \]
    proving the right equalities in Theorem \ref{theo:asum=prod}.
\end{proof}

\begin{rema}
    By eliminating the summation on $j$ in \eqref{eq:asum} using \refEulerr,
    we see
\begin{alignat*}{2}
    N_a
    &=
    \sum_{i\ge 0} (q^{1+2i+2t};q^2)_\infty
    \frac{q^{\binom{i}{2}+(1+s)i}}{(q;q)_i}
    & &=
    (q;q^2)_\infty
    \sum_{i\ge 0} \frac{q^{\binom{i}{2}+(1+s)i}}{(q;q)_i (q;q^2)_{i+t}}
\end{alignat*}
for $(a,s,t)=(1,0,0), (2,0,1), (3,1,1)$.
Hence, 
as Step 3 (in \S\ref{sect:intro:auto_qdiff}) for Nandi's conjectures,
we can employ instead of \eqref{eq:Slater:mod28}
another three identities
\cite{MR0049225}*{(81),(80),(82)}=\cite{MR3752624}*{(A.124),(A.125),(A.126)} (for the same $(a,s,t)$):
\[
    \sum_{i\ge 0} \frac{q^{\binom{i}{2}+(1+s)i}}{(q;q)_i (q;q^2)_{i+t}}
    =
    \frac{(q^{a},q^{7-a},q^7;q^7)_\infty (q^{7-2a},q^{7+2a};q^{14})_\infty}{(q;q)_\infty (q;q^2)_\infty}.
\]
\end{rema}

\appendix

\section{Textbook constructions for finite automata}
\label{sect:appe:FLT}

To recall the proof of Proposition \ref{prop:DFA:1}
we need \emph{$\vep$-NFAs}:

\begin{defi}[\cite{sipser13}*{Definition 1.37}]
    A \emph{nondeterministic finite automaton with $\varep$-transitions} 
    (or \emph{$\varep$-NFA} for short)
    over $\AL$ is a 5-tuple
    $\AUT=(Q,\AL,\De,s,F)$ where
    $Q$ is a finite set, 
    $\De\colon Q\times (\AL\sqcup\{\varep\})\lra 2^Q$, 
    $s\in Q$ and 
    $F\subset Q$. 
\end{defi}
\begin{defi}
    Let $\AUT=(Q,\AL,\De,s,F)$ be an $\varep$-NFA.

    \noindent (1)
    For $A\subset Q$,
    its \emph{$\vep$-closure} $\epcl(A)$ is
    the set of states that are reachable from a state in $A$
    via successive $\varep$-transitions,
    i.e.,\,$\cl(A):=\bigcup_{n\ge0} \De_\varep^n(A)$
    where $\De_\varep(B):=\bigcup_{q\in B}\De(q,\varep)$ for $B\subset Q$.
    
    \noindent (2)
    We
    define 
    $\hatt\De\colon Q\times \AL^*\lra 2^Q$
    inductively
    by $\hatt\De(q, \varepsilon)=\cl(\{q\})$ and
    $\hatt\De(q, wa)=\cl\big(\bigcup_{q'\in\hatt\De(q,w)} \De(q',a)\big)$ 
    ($q\in Q$, $a\in \AL$, $w\in \AL^*$).
    We write
    $L(\AUT)=\{w\in \AL^*\mid\hatt\De(s,w)\cap F\neq\emptyset\}$,
    the language that $\AUT$ \emph{recognizes}.
\end{defi}

\begin{prop}[See e.g.,\,\cite{sipser13}*{Corollary 1.40} for the details]
    \label{prop:DFA=eNFA}
    A language $X\subset\AL^*$ is regular if and only if
    there exists an $\varep$-NFA recognizing $X$.
\end{prop}
\begin{proof}
    Every DFA can be seen as an $\varep$-NFA (with no $\varep$-transitions).
    Conversely, an $\varep$-NFA $(Q,\AL,\De,s,F)$ 
    can be converted into an equivalent DFA 
    $(Q',\AL,\de',s',F')$ via the \emph{subset construction}:
    $Q'=2^Q$,
    \mbox{$\de'\colon Q'\times\AL\lra Q'\,;$} \mbox{$(A,a)\mapsto\cl(\bigcup_{q\in A}\De(q,a))$},
    $s'=\cl(\{s\})$ and
    $F'=\{A\subset Q\mid A\cap F\neq\emptyset\}$.
\end{proof}

\begin{proof}[Proof of Proposition \ref{prop:DFA:1}]
    Assume DFAs
    $(Q_1,\AL,\de_1,s_1,F_1)$ and
    $(Q_2,\AL,\de_2,s_2,F_2)$ recognize $Y$ and $Z$ respectively.
    By Proposition \ref{prop:DFA=eNFA}
    it suffices to give a DFA or an $\vep$-NFA recognizing
    (1) $Y\cap Z$,
    (2) $YZ$, 
    (3) $Y^c$,
    and
    (4) $Y^*$.

    \noindent (1)
    The DFA $(Q_1\times Q_2,\AL,\de,(s_1,s_2),F_1\times F_2)$ recognizes $Y\cap Z$,
    where $\de((q_1,q_2),a)=(\de_1(q_1,a),\de_2(q_2,a))$.

    \noindent (2)
    The $\vep$-NFA $(Q,\AL,\De,s,F)$ recognizes $YZ$, where
    $Q=Q_1\sqcup Q_2$, $s=s_1$, $F=F_2$,
    $\De(q,a)=\{\de_i(q,a)\}$ ($i=1,2$, $q\in Q_i$, $a\in\AL$) and
    $\De(q,\vep)=\{s_2\}$ if $q\in F_1$ and 
    $\De(q,\vep)=\emptyset$ if $q\in (Q_1\sm F_1)\sqcup Q_2$.

    \noindent (3)
    The DFA $(Q_1,\AL,\delta_1,s_1,Q_1\sm F_1)$ recognizes $Y^c$.

    \noindent (4)
    The $\vep$-NFA $(Q,\AL,\De,s,F)$ recognizes $Y^*$, where
    $Q=Q_1\sqcup\{s\}$,
    $F=\{s\}\sqcup F_1$,
    $\De(q,a)=\{\de_1(q,a)\}$ ($q\in Q_1$, $a\in\AL$),
    $\De(s,a)=\emptyset$ ($a\in\AL$),
    $\De(s,\vep)=\{s_1\}$,
    $\De(q,\vep)=\{s_1\}$ if $q\in F_1$ and
    $\De(q,\vep)=\emptyset$ if $q\in Q_1\sm F_1$.
\end{proof}

\begin{rema}[DFA minimization. See e.g.,\,\cite{MR1633052}*{Lecture 14}]
    \label{prop:DFA:2}
Given a DFA $M=(Q,\AL,\delta,s,F)$, 
one can compute 
$M_{\min}$ by the following algorithm.
    \begin{itemize}
        \item[1.] Remove all unreachable states.
        \item[2.] Mark all (unordered) pairs $\{q,q'\}$ with $q\in F$, $q'\in Q\sm F$.
        \item[3.] Repeat until no more changes occur: \\
            if there exists an unmarked pair $\{q,q'\}\subset Q$ 
            such that
            $\{\de(q,a),\de(q',a)\}$ is marked for some $a\in\AL$, 
            then mark $\{q,q'\}$.
        \item[4.] The relation ``$q\sim q':\Longleftrightarrow\{q,q'\}\text{ is unmarked}$''
            is then an equivalence relation.
            Writing $[q]:=\{q'\in Q\mid q\sim q'\}$,
            we have a new DFA $M_{\min}=(Q',\AL,\de',s',F')$ where
            $Q':=\{[q]\mid q\in Q\}$,
            $\de'([q],a):=[\de(q,a)]$,
            $s':=[s]$,
            $F':=\{[q]\mid q\in F\}$. 
            \qedhere
    \end{itemize}
\end{rema}

\section{Modified Murray--Miller Theorem}
\label{sect:MurrayMiller}

We review an algorithm
given in 
\cite{MR387178}*{p.1040},
\cite{MR0557013}*{Lemma 8.10}
(see also \cite{MR4072958}*{\S 3} for an exposition),
which obtain
a (nontrivial) $q$-difference equation for $F_1(x)$
from a given system of $q$-difference equations
\begin{equation}\label{eq:mqdiff}
    F_i(x) = \sum_{j=1}^{\ell} p_{ij}(x) F_j(xq^m)
    \qquad (i=1,\dots,\ell),
\end{equation}
where $p_{ij}(x)=p_{ij}(x,q)\in\bbQ(x,q)$.

\vspace{1mm}
\noindent\underline{Step 1}:
We obtain from \eqref{eq:mqdiff} another system
\begin{equation}\label{eq:mqdiff2}
    \begin{gathered}
        F'_i(x) = \sum_{j=1}^{\ell'} p'_{ij}(x) F'_j(xq^m)
        \quad (i=1,\dots,\ell'), 
    \end{gathered}
\end{equation}
where
$1\le \ell'\le\ell$,
$F'_1(x)=F_1(x)$, and
$(p'_{ij})_{i,j=1}^{\ell'}\in\mathrm{Mat}_{\l'}(\bbQ(x,q))$ is of the form 
\eqref{eq:sysqdiff:s,l} with $(s,\l)$ replaced by $(\l',\l')$.

Step 1 is done in Algorithm \ref{algo:sysqdiff:1},
which receives $(p_{ij}(x))_{i,j=1}^{\ell}$ as the input
and returns $(p'_{ij}(x))_{i,j=1}^{\ell'}$ as the output.
It
works as follows:
in the $s$-th iteration of the \textbf{for} loop in Algorithm \ref{algo:sysqdiff:1},
\begin{itemize}
    \setlength{\itemsep}{0mm}
    \item 
    in the line \ref{algo:line:assert},
    i.e.,\,at the beginning of the iteration,
it is ensured that (a) the matrix $P^{(s)}$ is defined and is of the form
\begin{equation}\label{eq:sysqdiff:s,l} 
    \text{\footnotesize
    \bordermatrix{
        & 1     & 2     &       & \cdots & s-1   & s    & \cdots & \ell \cr
        1 & \star & 1     & 0     & \cdots & 0     & 0    & \cdots & 0 \cr
        2 & \star & \star & 1     & \cdots & 0     & 0    &        & 0 \cr
    \vdots& \vdots& \vdots& \vdots& \ddots & \vdots&\vdots&        & \vdots \cr
        & \star & \star & \star & \cdots & 1     & 0    & \cdots & 0 \cr
        s-1& \star & \star & \star & \cdots & \star & 1    & \cdots & 0  \cr
        s & \star & \star & \star & \cdots & \star & \star& \cdots & \star \cr
    \vdots& \vdots& \vdots& \vdots& \vdots & \vdots&\vdots& \ddots & \vdots\cr
    \ell & \star & \star & \star & \cdots & \star & \star& \cdots & \star \cr
        },
    }
\end{equation}
    i.e.,\,$P^{(s)}_{1,2}=\cdots=P^{(s)}_{s-1,s}=1$ and
    $P^{(s)}_{i,j}=0$ if $i<s$ and $j>i+1$;
    and (b) $F'_1(x),\dots,F'_s(x)$ are (implicitly) defined 
    (we let $F'_1(x):=F_1(x)$ when $s=1$)
    and satisfy
    \begin{equation}\label{eq:algo:sysqdiff}
        \begin{aligned}
        &\transpose{\big(F'_{1}(x) , \cdots , F'_s(x) , F_{s+1}(x) , \cdots , F_\ell(x)\big)} \\
        &\quad= P^{(s)}\cdot
        \transpose{\big(F'_{1}(xq^m) , \cdots , F'_{s}(xq^m) , F_{s+1}(xq^m) , \cdots , F_\ell(xq^m)\big)}.
        \end{aligned}
    \end{equation}
    These assertions are obvious when $s=1$
    by putting $P^{(1)}:=(p_{ij}(x))_{i,j=1}^{\l}$.

    \item the \textbf{if} statement in the line 4 is always true if $s=\ell$.

    \item if the algorithm reaches the line \ref{algo:line:return},
        we can see 
        by \eqref{eq:sysqdiff:s,l} and \eqref{eq:algo:sysqdiff} that
        $F'_1(x),\dots,F'_{\ell'}(x)$ satisfy the system of $q$-difference equations
        \eqref{eq:mqdiff2} with 
        $(p'_{ij})_{i,j=1}^{\ell'}:=(P^{(s)}_{ij})_{i,j=1}^{s}$.

    \item the lines \ref{algo:line:switch} and \ref{algo:line:switch2} correspond to
        switching $F'_{s+1}(x)$ and $F'_t(x)$.
        To make the algorithm deterministic,
        one should choose the smallest $t$
        in the line \ref{algo:line:choose}, for example.

    \item in the line \ref{algo:line:P(s+1)},
        it can be checked (see \cite{MR4072958}*{Claim 3.1}) that
        $P^{(s+1)}\,(\in\mathrm{Mat}_{\l}(\bbQ(x,q)))$ 
        is again of the form \eqref{eq:sysqdiff:s,l} with $(s,\l)$ replaced by $(s+1,\l)$.
        Moreover, for
        $F'_{s+1}(x) := \sum_{j=s+1}^{\ell} P^{(s)}_{s,j}(xq^{-m}) F_j(x)$
        we can check \eqref{eq:algo:sysqdiff} with $s$ replaced by $s+1$
        (see \cite{MR4072958}*{(3.7)}).
\end{itemize}

\begin{algorithm}
\algsetup{indent=2em}    
\begin{algorithmic}[1]
    \REQUIRE $(p_{ij}(x))_{i,j=1}^{\ell}$
    \COMMENT the coefficients in \eqref{eq:mqdiff}
    \ENSURE $\ell'$, $(p'_{ij}(x))_{i,j=1}^{\ell'}$ 
    \COMMENT the coefficients in \eqref{eq:mqdiff2}
    \STATE $P^{(1)}\leftarrow (p_{ij})_{i,j=1}^{\ell}$
    \FOR{$s=1$ \TO $\ell$}
        \STATE 
        \label{algo:line:assert}
        \COMMENT{assert that $P^{(s)}$ is of the form 
            \eqref{eq:sysqdiff:s,l}
        }
        \IF{$P^{(s)}_{s,s+1}=P^{(s)}_{s,s+2}=\cdots=P^{(s)}_{s,\ell}=0$} 
            \RETURN $s$, $\big(P^{(s)}_{ij}(x)\big)_{i,j=1}^{s}$ 
            \label{algo:line:return}
        \ENDIF
        \IF{$P^{(s)}_{s,s+1}=0$}
            \STATE choose any $t$ such that $s+1<t\le\ell$ and $P^{(s)}_{s,t}\neq 0$
                \label{algo:line:choose}
            \STATE swap $s+1$-th and $t$-th rows of $P^{(s)}$
                \label{algo:line:switch}
            \STATE swap $s+1$-th and $t$-th columns of $P^{(s)}$
                \label{algo:line:switch2}
        \ENDIF
        \STATE $T_s(x) \leftarrow$
            {\footnotesize
            $\bordermatrix{
                    & 1 &        & s & s+1  &  &  & \ell \cr
                1 & 1 & \cdots & 0 & 0    & 0 & \cdots & 0 \cr
                    & \vdots & \ddots & \vdots &\vdots& \vdots & \ddots & \vdots \cr
                  s & 0 & \cdots  & 1 & 0 & 0 & \cdots & 0 \cr
                s+1 & 0      & \cdots & 0      & P^{(s)}_{s,s+1}(x) & P^{(s)}_{s,s+2}(x) & \cdots & P^{(s)}_{s,\ell}(x) \cr
                    & 0      & \cdots & 0      & 0 & 1 & \cdots & 0 \cr
                    & \vdots & \ddots & \vdots & \vdots & \vdots & \ddots & \vdots \cr
                \ell & 0      & \cdots & 0      & 0 & 0 & \cdots & 1
            }$}
            \label{algo:line:Ts}
        \STATE $P^{(s+1)} \leftarrow T_s(xq^{-m}) P^{(s)} T_s(x)^{-1}$ \label{algo:line:P(s+1)}
    \ENDFOR
\end{algorithmic}
\caption{Obtain \eqref{eq:mqdiff2} from \eqref{eq:mqdiff}
(\cite{MR0557013}*{Lemma 8.10}, see also \cite{MR4072958}*{\S 3})
}
\label{algo:sysqdiff:1}
\end{algorithm}

This completes the algorithm to obtain a new system \eqref{eq:mqdiff2}.

\vspace{1mm}
\noindent\underline{Step 2}:
Now that
the $i$-th equation (for $i=1,\cdots,\ell'-1$) in \eqref{eq:mqdiff2} is of the form
$0 = - F'_{i}(x) + F'_{i+1}(xq^m) + \sum_{j<i+1}p'_{ij}(x) F'_{j}(xq^m)$,
we can eliminate $F'_{\l'},\cdots,F'_{2}$ from the system (in this order)
to transform the final equation in \eqref{eq:mqdiff2} 
into 
a $q$-difference equation for $F'_1(x)=F_1(x)$,
which is nontrivial
(see \cite{MR0557013}*{Lemma 8.10} for more details).

\section{Proof of Proposition \ref{theo:qd}} \label{sect:qd:proof}

We apply the algorithm in Appendix \ref{sect:MurrayMiller} 
to \eqref{eq:nandi:sysqdiff}.

\subsection{The case $\cN_1$}

To find a $q$-difference equaiton for $F_{7}(x)$,
first we permute the positions of $F_0,\dots,F_5,F_7$
in \eqref{eq:nandi:sysqdiff} as follows:

{\footnotesize
\[
    \begin{pmatrix}
        F_{7}(x) \\ 
        F_{1}(x) \\ 
        F_{2}(x) \\ 
        F_{3}(x) \\ 
        F_{4}(x) \\ 
        F_{5}(x) \\ 
        F_{0}(x)
    \end{pmatrix}
    = 
    \left(\begin{array}{rrrrrrr}
        0 & x q^{2} & x^{2} q^{4} & 0 & 0 & 0 & 1 \\
        0 & x q^{2} & 0 & 0 & 0 & 1 & 0 \\
        1 & 0 & 0 & 0 & 0 & 0 & 0 \\
        0 & x q^{2} & 0 & x q & 0 & 1 & 0 \\
        1 & 0 & 0 & 0 & x q^{2} & 0 & 0 \\
        0 & x q^{2} & x^{2} q^{4} & x q & 0 & 0 & 1 \\
        0 & x q^{2} & x^{2} q^{4} & x q & x^{2} q^{2} & 0 & 1
        \end{array}\right)
    \begin{pmatrix}
        F_{7}(xq^2) \\ 
        F_{1}(xq^2) \\ 
        F_{2}(xq^2) \\ 
        F_{3}(xq^2) \\ 
        F_{4}(xq^2) \\ 
        F_{5}(xq^2) \\ 
        F_{0}(xq^2) 
    \end{pmatrix}.
\]
}

Next we apply Algorithm \ref{algo:sysqdiff:1}
(note that it is deterministic).
It stops at
the 5-th iteration of the \textbf{for} loop and we get
{\footnotesize
    \[
\begin{pmatrix}
        G_{1}(x) \\ 
        G_{2}(x) \\ 
        G_{3}(x) \\ 
        G_{4}(x) \\ 
        G_{5}(x) \\ 
        G_{6}(x) \\ 
        G_{7}(x)
    \end{pmatrix}
 = \left(\begin{array}{rrrrrrr}
    0 & 1 & 0 & 0 & 0 & 0 & 0 \\
    x^{2} & x + 1 & 1 & 0 & 0 & 0 & 0 \\
    \frac{-x^{3} + x^{2}}{q^{2}} & \frac{x}{q} & \frac{1}{q} & 1 & 0 & 0 & 0 \\
    \frac{-x^{2}}{q^{3}} & \frac{-x q^{2} + x^{2}}{q^{4}} & \frac{-q^{2} + x}{q^{4}} & \frac{-q^{2} + x}{q^{3}} & 1 & 0 & 0 \\
    \frac{-x^{3}}{q^{7}} & 0 & 0 & 0 & \frac{x}{q^{4}} & 0 & 0 \\
    0 & 1 & 0 & \frac{q}{x - 1} & \frac{q^{3}}{-x^{2} + x} & 0 & 0 \\
    0 & 1 & 0 & \frac{q}{x - 1} & \frac{q^{3}}{-x + 1} & 0 & 0
    \end{array}\right)
\begin{pmatrix}
        G_{1}(xq^2) \\ 
        G_{2}(xq^2) \\ 
        G_{3}(xq^2) \\ 
        G_{4}(xq^2) \\ 
        G_{5}(xq^2) \\ 
        G_{6}(xq^2) \\ 
        G_{7}(xq^2)
    \end{pmatrix}, \]
}

\noindent
where the middle matrix is $P^{(5)}$ in the notation of Algorithm \ref{algo:sysqdiff:1}
and each $G_i$ is a certain $\mathbb{Q}(x,q)$-linear combination of $F_j$ ($j\in\{0,\dots,5,7\})$
with $G_1=F_7$. 
The equation given in the $i$-th row ($i=1,\cdots,4$) is 
\[
0 = -G_{i}(x) + G_{i+1}(xq^2) + \sum_{j\leq i}P^{(5)}_{i,j}(x) G_{j}(xq^2), 
\]
by which 
each 
$G_{i+1}(x)$ is written in terms of 
$G_j(x)$ ($j\le i$) and $G_{i}(xq^{-2})$.
Thus we can eliminate $G_5,\dots,G_2$ 
and then the equation in the 5-th row
\[
0=
-G_5(x)-\frac{x^3}{q^{7}} G_1(xq^2)+\frac{x}{q^4} G_5(xq^2)
\]
turns into an equation 
only regarding $\{G_1(xq^{2k})\mid k\in\mathbb{Z}\}$:
{\footnotesize 
\begin{align}
    0 =&  
    - G_1(xq^ { -8 }) 
    + \frac{q^{6} + x (q^{2} + q + 1)}{q^{6}} G_1(xq^ { -6 }) 
    - \frac{x q^{8} + x^{2} (q^{5} + q^{4} + q^{3} - 1)}{q^{12}} G_1(xq^ { -4 })
    \\
    &- \frac{- x^{2} q^{4} - x^{3} (q^{5} + q^{2} + q + 1)}{q^{14}} G_1(xq^ { -2 }) 
    + \frac{x^3(x-q^2)(1+q+q^2)}{q^{13}} G_1(x) 
    \\
    &- \frac{x^3(x-1)(x-q^2)}{q^{9}} G_1(xq^ { 2 }) .
    \label{eq:qd:elim:4}
\end{align}
}
    
\noindent
By letting $x\mapsto xq^{8}$
in \eqref{eq:qd:elim:4},
we obtain \eqref{eq:qd:F} for 
$G_1(x)=F_{7}(x)=f_{\cN_1}(x,q)$.

\subsection{The case $\cN_2$}

The proof of Proposition \ref{theo:qd} for $\cN_2$ (and $\cN_3$)
proceeds almost the same: 
we start by rewriting \eqref{eq:nandi:sysqdiff} as

\[
    \footnotesize 
    \begin{pmatrix}
        F_{3}(x) \\ 
        F_{1}(x) \\ 
        F_{2}(x) \\ 
        F_{4}(x) \\ 
        F_{7}(x) \\ 
        F_{5}(x) \\ 
        F_{0}(x) \\ 
    \end{pmatrix}
    = 
    \left(\begin{array}{rrrrrrr}
        x q & x q^{2} & 0 & 0 & 0 & 1 & 0 \\
        0 & x q^{2} & 0 & 0 & 0 & 1 & 0 \\
        0 & 0 & 0 & 0 & 1 & 0 & 0 \\
        0 & 0 & 0 & x q^{2} & 1 & 0 & 0 \\
        0 & x q^{2} & x^{2} q^{4} & 0 & 0 & 0 & 1 \\
        x q & x q^{2} & x^{2} q^{4} & 0 & 0 & 0 & 1 \\
        x q & x q^{2} & x^{2} q^{4} & x^{2} q^{2} & 0 & 0 & 1
        \end{array}\right)
    \begin{pmatrix}
        F_{3}(xq^2) \\ 
        F_{1}(xq^2) \\ 
        F_{2}(xq^2) \\ 
        F_{4}(xq^2) \\ 
        F_{7}(xq^2) \\ 
        F_{5}(xq^2) \\ 
        F_{0}(xq^2) \\ 
    \end{pmatrix}.
\]
\noindent
Here we permuted the positions of $F_0,\dots,F_5,F_7$
so that
further row (and column) swapping
(in the lines \ref{algo:line:switch} and \ref{algo:line:switch2}
of Algorithm \ref{algo:sysqdiff:1}) will not happen.
Then Algorithm \ref{algo:sysqdiff:1} stops at the 5-th iteration with

{\footnotesize 
\[P^{(5)} = 
    \left(\begin{array}{rrrrrrr}
    x q & 1 & 0 & 0 & 0 & 0 & 0 \\
    x q & x + 1 & 1 & 0 & 0 & 0 & 0 \\
    0 & 0 & 0 & 1 & 0 & 0 & 0 \\
    0 & \frac{x^{2}}{q^{4}} & \frac{x^{2}}{q^{4}} & \frac{x}{q^{2}} & 1 & 0 & 0 \\
    0 & \frac{x^{2} q^{2} - x^{3}}{q^{8}} & \frac{x^{2} q^{2} - x^{3}}{q^{8}} & 0 & 0 & 0 & 0 \\
    x q & 1 & 1 & 0 & 0 & 0 & 0 \\
    x q & 1 & 1 & 1 & \frac{q^{2}}{x - 1} & 0 & 0
    \end{array}\right),
\]
}

\noindent
and by the same procedure
we obtain
{\footnotesize
\begin{align}
0 =
& - G_1(xq^ { -8 }) 
 + \frac{q^{7} + x(1+q+q^{3})}{q^{7}} G_1(xq^ { -6 })
 - \frac{ x q^{7} + x^{2} q^{2} + x^{2}}{q^{11}} G_1(xq^ { -4 }) \\
&+ \frac{x^{2} q^{4} -  x^{3} q^{2} -  x^{3}}{q^{10}} G_1(xq^ { -2 }) 
 + \frac{x^3(x-q^2)(1+q^2+q^3)}{q^{11}} G_1(x) \\
&- \frac{x^3(x-1)(x-q^2)}{q^{7}} G_1(xq^ { 2 }),
\label{eq:qd146:qd:elim:4}
\end{align}  
}

\noindent
where $G_1(x)=F_{3}(x)=f_{\cN_2}(x,q)$.
Now,
by letting $x\mapsto xq^{8}$ in
\eqref{eq:qd146:qd:elim:4}
we obtain 
\eqref{eq:qd:F} for 
$f_{\cN_2}(x,q)$.

\subsection{The case $\cN_3$}

Similarly, we start the algorithm by writing

{\footnotesize
\[
    \begin{pmatrix}
        F_{4}(x) \\ 
        F_{7}(x) \\ 
        F_{2}(x) \\ 
        F_{3}(x) \\ 
        F_{5}(x) \\ 
        F_{1}(x) \\ 
        F_{0}(x)
    \end{pmatrix}
    = 
    \left(\begin{array}{rrrrrrr}
        x q^{2} & 1 & 0 & 0 & 0 & 0 & 0 \\
        0 & 0 & x^{2} q^{4} & 0 & 0 & x q^{2} & 1 \\
        0 & 1 & 0 & 0 & 0 & 0 & 0 \\
        0 & 0 & 0 & x q & 1 & x q^{2} & 0 \\
        0 & 0 & x^{2} q^{4} & x q & 0 & x q^{2} & 1 \\
        0 & 0 & 0 & 0 & 1 & x q^{2} & 0 \\
        x^{2} q^{2} & 0 & x^{2} q^{4} & x q & 0 & x q^{2} & 1
        \end{array}\right)
    \begin{pmatrix}
        F_{4}(xq^2) \\ 
        F_{7}(xq^2) \\ 
        F_{2}(xq^2) \\ 
        F_{3}(xq^2) \\ 
        F_{5}(xq^2) \\ 
        F_{1}(xq^2) \\ 
        F_{0}(xq^2)
    \end{pmatrix}.
\]
}

\noindent
Then Algorithm \ref{algo:sysqdiff:1} stops at the 5-th iteration with
{\footnotesize
\[P^{(5)} = \left(\begin{array}{rrrrrrr}
    x q^{2} & 1 & 0 & 0 & 0 & 0 & 0 \\
    0 & 0 & 1 & 0 & 0 & 0 & 0 \\
    x^{2} q^{2} & x^{2} & 1 & 1 & 0 & 0 & 0 \\
    0 & 0 & \frac{x}{q^{2}} & \frac{x q + x}{q^{2}} & 1 & 0 & 0 \\
    0 & 0 & \frac{x q^{2} - x^{2}}{q^{5}} & \frac{x q^{2} - x^{2}}{q^{5}} & 0 & 0 & 0 \\
    0 & 0 & 0 & 0 & \frac{q}{-x^{2} + x} & 0 & 0 \\
    x^{2} q^{2} & 0 & 1 & 1 & \frac{-q}{-x + 1} & 0 & 0
    \end{array}\right). \]
}

\noindent
By the same procedure
we obtain
{\footnotesize
\begin{align}
    0 =
    &- G_1(xq^ { -8 }) 
    + \frac{q^{6} + x (q^{3} + q^{2} + 1)}{q^{6}} G_1(xq^ { -6 }) 
    - \frac{x^{2} q^{2} + x q^{3} + x^{2}}{q^{7}} G_1(xq^ { -4 }) 
    \\
    &+ \frac{x^{2} q^{4} - x^{3} q^{2} - x^{3}}{q^{10}} G_1(xq^ { -2 }) 
    + \frac{x^3(x-q^2)(1+q+q^3)}{q^{8}} G_1(x) 
    \\ 
    &- \frac{x^3(x-1)(x-q^2)}{q^{3}} G_1(xq^ { 2 }),
    \label{eq:qd256:elim:5}
\end{align}
}

\noindent
where $G_1(x)=F_{4}(x)=f_{\cN_3}(x,q)$.
Now,
by letting $x\mapsto xq^{8}$ in
\eqref{eq:qd256:elim:5}
we obtain 
\eqref{eq:qd:F} for $f_{\cN_3}(x,q)$.

\section{Minimal forbidden patterns and prefixes}
\label{sect:appe:minFP}

Let $\AL$ be a nonempty finite set.
For $B\subset\AL^*$
the language 
$\AL^*B\AL^*$ 
(resp.\,$B\AL^*$)
consists of words
matching
(resp.\,beginning with) some $w\in B$
(see Definition \ref{defi:match};
here we permit the case $n=0$),
and a language $A\subset\AL^*$ is of the form
$A=\AL^*B\AL^*$ 
(resp.\,$A=B\AL^*$)
for some $B\subset\AL^*$
if and only if 
$A=\AL^*A\AL^*$
(resp.\,$A=A\AL^*$).
In \cite{MR1638178} they gave an algorithm to find
from given
$A=\AL^*A\AL^*\subset\AL^*$
the minimum $B\subset\AL^*$ such that 
$A=\AL^*B\AL^*$.
By a slight generalization
it can also be used to find the minimum $B$ for which
$A=B\AL^*$, given $A\subset\AL^*$ such that $A=A\AL^*$
(Proposition \ref{theo:min_prefix} and \ref{theo:minimalA}).

In general,
for a poset $(P,\le)$ and a subset $A\subset P$
we write
\begin{align}
    \Base A
    = \Base_{\le} A  
    &:= \{w\in A\mid \forall v\in A, (v\le w\Lra v=w)\},
    \\
    \OF(A)
    = \OF_{\le} (A)
    &:= \{w\in P\mid \exists v\in A,\,v\le w\}.
\end{align}
Let us say a poset
$(P,\le)$ is \emph{good}
if $A\subset\OF(\Base A)$ for any $A\subset P$.

\begin{prop}\label{theo:min_prefix}
    Let $(P,\le)$ be a good poset.
    For $A=\OF(A)\,(\subset P)$
    and $B\subset P$, it holds
    $A=\OF(B)$ if and only if $\Base A\subset B\subset A$.
\end{prop}

\begin{proof}
    {($\Longrightarrow$)}:
    Assume $A=\OF(B)$.
    Then $B\subset A$ is obvious.
    For any $w\in A$
    we have $u\le w$ for some $u\in B$,
    and
    if $w\in \Base A$
    then $w=u$.    
    {($\Longleftarrow$)}:
    $\Base A\subset B\subset A$ implies
    $A\subset \OF(\Base A)\subset \OF(B)\subset \OF(A)=A$.
\end{proof}

Let us consider 
partial orders
$\le$ and $\le_r$ on $\AL^*$ defined by
\begin{align}
    v\le w
    &:\Llra
    \exists u\in\AL^*,\exists u'\in\AL^*,\,w=uvu',
    \\
    v\le_r w
    &:\Llra
    \exists u\in\AL^*,\,w=vu.
\end{align}
Clearly 
$\OF_{\le}(B) = \AL^*B\AL^*$ and
$\OF_{\le_r}(B) = B\AL^*$
for $B\subset\AL^*$.
It is easy to see that
$(\AL^*,\le)$ and 
$(\AL^*,\le_r)$ are good.

\begin{prop}\label{theo:minimalA}
    Let $A\subset\AL^+\,(=\AL^*\sm\{\varep\})$.

    \noindent $(1)$
    If $A=\AL^*A\AL^*\,(=\OF_{\le}(A))$ then
    $\Base_{\le} A = A\cap A^c\AL \cap \AL A^c$.

    \noindent $(2)$
    If $A=A\AL^*\,(=\OF_{\le_r}(A))$ then
    $\Base_{\le_r} A = A\cap A^c\AL$.
\end{prop}
\begin{proof}
    (1) is \cite{MR1638178}*{Eq.\,(2)}
    with $A$ replaced by $A^c$.
    (2) is proved parallelly, but for completeness we duplicate a proof.
    \underline{($\subset$)}: 
    Clearly $\Base_{\le_r} A\subset A$.
    For any $w\in \Base_{\le_r} A$, since $w\neq\varep$ (otherwise we get $A=\AL^*$)
    we can write $w=w'a$ with $w'\in\AL^*$, $a\in\AL$.
    Then $w'\notin A$ by $w\in \Base_{\le_r} A$. Hence $w=w'a\in A^c\AL$.
    \underline{($\supset$)}: 
    For any $w=a_1\cdots a_n\in A^c\AL$,
    it holds $n\ge1$ and $a_1\cdots a_{n-1}\notin A$.
    For $v\in\AL^*$, 
    if $v<w$ then
    $v=a_1\cdots a_i$
    for some $0\le i < n$,
    and hence $v\notin A$
    since $A=A\AL^*$.
    Therefore $w\in\Base_{\le_r} A$ if $w\in A$.
\end{proof}

\begin{lemm}\label{theo:X'0}
    Let $A,X\subset\AL^+$
    and assume
    $\AL^*X\AL^*\subset A=A\AL^*$.
    Then
    \begin{equation}\label{eq:Xv2}
        X':=
        \Base_{\le_r}(A)\sm \AL^*X
        =
        \big(A\cap (A^c\,\AL)\big) \sm \AL^*X \quad(\subseteq \AL^+).
    \end{equation}
    is the minimum set 
    (with respect to inclusion)
    such that
    $A=\AL^*X\AL^*\cup X'\AL^*$.
\end{lemm}
\begin{proof}
    The right equality in \eqref{eq:Xv2} follows from
    Proposition \ref{theo:minimalA} (2).
    We apply 
    Proposition \ref{theo:min_prefix}:
    writing
    $A'=\AL^*X$,
    we have
    $A=A'\AL^*\cup B\AL^*
    \,(=\OF_{\le_r}(A'\cup B)) 
    \Llra \Base_{\le_r} A\subset (A'\cup B)\subset A
    \Llra (\Base_{\le_r} A)\sm A'\subset B\subset A
    $ for $B\subset\AL^*$. Thus, $X'$ is the desired one.
\end{proof}

We apply this to DFAs.
Recall Definition \ref{eq:Mv}.

\begin{prop}\label{theo:L(Mv)2}
    Let $M=(Q,\AL,\delta,s,F)$ be a DFA
    and assume $L(M)=\AL^*X\AL^*\cup X'\AL^*$
    for some $X,X'\subset\AL^*$.
    For any reachable state $v\in Q\sm F$ it holds
    $L(M_v)=\AL^*X\AL^*\cup X_v\AL^*$,
    where 
    \begin{equation}\label{eq:Xv}
        X_v:=
        \big(L(M_v)\cap (L(M_v)^c\,\AL)\big) \sm \AL^*X \quad(\subseteq \AL^+).
    \end{equation}
    Moreover, 
    $X_v$ is the minimum such set 
    (with respect to inclusion).
\end{prop}
\begin{proof}
    By the reachability,
    $\hatt\delta(s,b)=v$
    for some $b\in \AL^*$.
    Then 
    $a\in L(M_v)\Llra ba\in L(M)=\AL^*X\AL^*\cup X'\AL^*$
    for any $a\in\AL^*$,
    by which
    $\AL^*\FB\AL^*\subset L(\AUT_v)=L(\AUT_v)\AL^*$
    follow.
    Now the proposition follows from
    Lemma \ref{theo:X'0} 
    (note that $v\notin F$ implies
    $\varep\notin L(M_v)$).
\end{proof}

\section{Connection to linked partition ideals}
\label{sect:appe:LPI}

\subsection{On the definition of partition ideals}
\label{sect:appe:PI}

Consider a partial order $\le$ on $\Par\,(\simeq\hPar)$
defined by
$(f_i)_{i\ge 1}\le (g_i)_{i\ge 1}:\Llra\forall i\ge1,f_i\le g_i$.
In \cite{MR387178}*{Definition 1} 
a subset $\cC$ of $\Par$ is called a \emph{partition ideal} 
(\emph{PI} for short)
if it is an order ideal (\cite{MR2868112}*{p.282}) with respect to $\le$, i.e., 
\begin{equation}\label{eq:ParIdeal}
    \forall f\in\hatt\cC,
    \,\forall g\in\hPar,
    \,(g\le f
    \implies g\in\hatt\cC).
\end{equation}

For $m>0$ and $\la\in\Par$
we write
$\la_{>m} := (\la_1,\dots,\la_{\l'})$
where $\l':=\#\{i\ge 1\mid \la_i>m\}$.
In \cite{MR387178}*{Definition 7}
a PI $\cC$ is defined to have \emph{modulus} $m>0$ if
$\rsh^m(\cC) = \cC_{>m} :=\{\la\in\cC\mid\la=\la_{>m}\}$.
As we see below,
this is equivalent 
to adding an extra condition $\rsh^m(\cC)\subset\cC$
to $\lsh^m(\cC)\subset\cC$
(cf.\,\eqref{eq:mWPImodulus})
under the assumption
\begin{equation}\label{eq:PI2}
    \la\in\cC\implies \la_{>m} \in\cC.
\end{equation}

\begin{prop}\label{theo:modulus_cond}
    For any subset $\cC\subset\Par$ satisfying
    \eqref{eq:PI2},
    it holds
    $\rsh^m(\cC) = \cC_{>m}$
    if and only if
    $\rsh^m(\cC)\subset\cC$
    and
    $\lsh^m(\cC)\subset\cC$.
\end{prop}

\begin{proof}
    \noindent\underline{($\Rightarrow$)}:
    Assume
    $\rsh^m(\cC) = \cC_{>m}$.
    Then
    obviously
    $\rsh^m(\cC)\subset\cC$.
    Since $\la_{>m}=\rsh^m\lsh^m(\la)$ for any $\la\in\Par$,
    \eqref{eq:PI2} implies
    $\rsh^m\lsh^m(\cC)\subset\cC_{>m}$ ($=\rsh^m(\cC)$),
    and hence
    $\lsh^m(\cC)\subset\cC$
    since $\rsh$ is injective.
    \underline{($\Leftarrow$)}:
    Assume
    $\rsh^m(\cC)\subset\cC$
    and
    $\lsh^m(\cC)\subset\cC$.
    Then obviously
    $\rsh^m(\cC)\subset\cC_{>m}$.
    Since $\lsh^m(\cC)\subset\cC$ we have
    $\rsh^m\lsh^m(\cC_{>m})\subset\rsh^m(\cC)$,
    and
    $\cC_{>m}=\rsh^m\lsh^m(\cC_{>m})$
    since $\rsh^m\lsh^m$ is identical on $\Par_{>m}$.
    Hence $\cC_{>m}\subset\rsh^m(\cC)$.
\end{proof}

\begin{coro}
    A PI having modulus $m$ satisfies \eqref{eq:mWPImodulus}.
\end{coro}
\begin{proof}
    Since a PI satisfies \eqref{eq:PI2}
    we can apply Proposition \ref{theo:modulus_cond}.
    The left condition in \eqref{eq:mWPImodulus} is obvious
    from \eqref{eq:ParIdeal}.
\end{proof}

\subsection{Linked partition ideals}\label{sect:lpi:lpi}
Recall Definition \ref{defi:pidata}.

\begin{defi}[\cite{MR387178}*{Definition 11}]\label{defi:lpi}
    A subset $\cC$ of $\Par$ is a 
    \emph{linked partition ideal} 
    (\emph{LPI} for short)
    if there exists $m\in\bbZ_{>0}$ for which
    \begin{enumerate}
        \setlength{\itemsep}{0mm}
        \item[(L1)] $\cC$ is a PI having modulus $m$;
        \item[(L2)] $|\cC_{\le m}|<\infty$;
        \item[(L3)]
            there exist
            $L\colon \cC_{\le m}\lra 2^{\cC_{\le m}}$
            and
            $\spa\colon \cC_{\le m}\lra\bbZ_{>0}$
            such that
            $\id^\pst_{\cC_{\le m}} (S) = \cC$,
            where $S$ is the set of
            $(\la^{(i)})_{i\ge1}\in \Seq(\cC_{\le m},\id_{\cC_{\le m}})$ with
            \begin{equation}\label{eq:lpi_orig}
                \forall j\ge 1,\,
                \la^{(j+1)}=\cdots\la^{(j+\spa(\la^{(j)})-1)} = \emptypar
                \text{ and }
                \la^{(j+\spa(\la^{(j)}))}\in L(\la^{(j)}).
            \end{equation}
    \end{enumerate}
\end{defi}

\begin{prop}\label{exam:LPI}
    An LPI $\cC$
    is regularly linked
    (see Definition \ref{defi:regularWPI}).
\end{prop}
\begin{proof}
    If $\cC=\emptyset$ then we can take
    $m,I,\pi,X'$ arbitrarily and $X=I$
    in Definition \ref{defi:regularWPI}.
    Assume $\cC\neq\emptyset$
    and let $m$ be as in Definition \ref{defi:lpi}.
    Since $\cC$ is a PI,
    we have $\emptypar\in\cC$
    and in particular $\cC_{\le m}\neq\emptyset$.
    Write $I:=\cC_{\le m}$ and $\pi:=\id_{\cC_{\le m}}$.
    Then the set $S\subset\Seq(I,\pi)$ in (L3)
    can be written as
    \[
        \Avoid{\Seq(I,\pi)}{X}{\emptyset}
        = S,
    \]
    where 
    \[
        X := \bigcup_{\la\in I} 
        \{\la\} \Big(I^{\spa(\la)} \sm \big(
            \{\underbrace{\emptypar\cdots\emptypar}_{\spa(\la)-1}\} L(\la) 
        \big) \Big)
        \quad(\subset I^+),
    \]
    which is finite and hence is a regular language over $I$
    (recall \eqref{eq:reg_op}).
\end{proof}

\subsection*{Acknowledgments}
We thank S.~Kanade and M.~Russell for helpful
discussions.
M.T.\,was supported by Start-up research support from Okayama University.
S.T.\,was supported by JSPS KAKENHI Grant 17K14154, 20K03506
and by Leading Initiative for Excellent Young Researchers, MEXT, Japan.

\bibliographystyle{abbrv}

\begin{bibdiv}
\begin{biblist}

\bib{MR225741}{article}{
      author={Andrews, George~E.},
       title={On partition functions related to {S}chur's second partition
  theorem},
        date={1968},
        ISSN={0002-9939},
     journal={Proc. Amer. Math. Soc.},
      volume={19},
       pages={441\ndash 444},
         url={https://doi.org/10.2307/2035546},
}

\bib{MR351985}{article}{
      author={Andrews, George~E.},
       title={An analytic generalization of the {R}ogers-{R}amanujan identities
  for odd moduli},
        date={1974},
        ISSN={0027-8424},
     journal={Proc. Nat. Acad. Sci. U.S.A.},
      volume={71},
       pages={4082\ndash 4085},
         url={https://doi.org/10.1073/pnas.71.10.4082},
}

\bib{MR387178}{article}{
      author={Andrews, George~E.},
       title={A general theory of identities of the {R}ogers-{R}amanujan type},
        date={1974},
        ISSN={0002-9904},
     journal={Bull. Amer. Math. Soc.},
      volume={80},
       pages={1033\ndash 1052},
         url={https://doi.org/10.1090/S0002-9904-1974-13616-5},
}

\bib{MR0557013}{book}{
      author={Andrews, George~E.},
       title={The theory of partitions},
   publisher={Addison-Wesley Publishing Co., Reading, Mass.-London-Amsterdam},
        date={1976},
        note={Encyclopedia of Mathematics and its Applications, Vol. 2},
}

\bib{MR1284057}{incollection}{
      author={Andrews, George~E.},
       title={Schur's theorem, {C}apparelli's conjecture and {$q$}-trinomial
  coefficients},
        date={1994},
   booktitle={The {R}ademacher legacy to mathematics ({U}niversity {P}ark,
  {PA}, 1992)},
      series={Contemp. Math.},
      volume={166},
   publisher={Amer. Math. Soc., Providence, RI},
       pages={141\ndash 154},
         url={https://doi.org/10.1090/conm/166/01622},
}

\bib{MR541344}{article}{
      author={Bressoud, David~M.},
       title={A generalization of the {R}ogers-{R}amanujan identities for all
  moduli},
        date={1979},
        ISSN={0097-3165},
     journal={J. Combin. Theory Ser. A},
      volume={27},
      number={1},
       pages={64\ndash 68},
         url={https://doi.org/10.1016/0097-3165(79)90008-6},
}

\bib{MR556608}{article}{
      author={Bressoud, David~M.},
       title={Analytic and combinatorial generalizations of the
  {R}ogers-{R}amanujan identities},
        date={1980},
        ISSN={0065-9266},
     journal={Mem. Amer. Math. Soc.},
      volume={24},
      number={227},
       pages={54},
         url={https://doi.org/10.1090/memo/0227},
}

\bib{1809.06089}{article}{
      author={Bringmann, Kathrin},
      author={Jennings-Shaffer, Chris},
      author={Mahlburg, Karl},
       title={Proofs and reductions of various conjectured partition identities
  of {K}anade and {R}ussell},
     journal={J. Reine Angew. Math.},
        note={to appear},
}

\bib{MR2637254}{book}{
      author={Capparelli, Stefano},
       title={Vertex operator relations for affine algebras and combinatorial
  identities},
   publisher={ProQuest LLC, Ann Arbor, MI},
        date={1988},
         url={http://gateway.proquest.com/openurl?url_ver=Z39.88-2004&
  rft_val_fmt=info:ofi/fmt:kev:mtx:dissertation&res_dat=xri:
  pqdiss&rft_dat=xri:pqdiss:8914197},
        note={Thesis (Ph.D.)--Rutgers The State University of New Jersey - New
  Brunswick},
}

\bib{MR1333389}{article}{
      author={Capparelli, Stefano},
       title={A construction of the level {$3$} modules for the affine {L}ie
  algebra {$A^{(2)}_2$} and a new combinatorial identity of the
  {R}ogers-{R}amanujan type},
        date={1996},
        ISSN={0002-9947},
     journal={Trans. Amer. Math. Soc.},
      volume={348},
      number={2},
       pages={481\ndash 501},
         url={https://doi.org/10.1090/S0002-9947-96-01535-8},
}

\bib{MR4072958}{article}{
      author={Chern, Shane},
      author={Li, Zhitai},
       title={Linked partition ideals and {K}anade-{R}ussell conjectures},
        date={2020},
        ISSN={0012-365X},
     journal={Discrete Math.},
      volume={343},
      number={7},
       pages={111876, 24},
         url={https://doi.org/10.1016/j.disc.2020.111876},
}

\bib{MR1638178}{article}{
      author={Crochemore, M.},
      author={Mignosi, F.},
      author={Restivo, A.},
       title={Automata and forbidden words},
        date={1998},
        ISSN={0020-0190},
     journal={Inform. Process. Lett.},
      volume={67},
      number={3},
       pages={111\ndash 117},
         url={https://doi.org/10.1016/S0020-0190(98)00104-5},
}

\bib{GAP_Automata1.14}{misc}{
      author={Delgado, M.},
      author={Linton, S.},
      author={Morais, J.~J.},
       title={{Automata}, a package on automata, {V}ersion 1.14},
         how={\href {https://gap-packages.github.io/automata/}
  {\texttt{https://gap-packages.github.io/}\discretionary
  {}{}{}\texttt{automata/}}},
        date={2018},
        note={Refereed GAP package},
}

\bib{MR0089895}{book}{
      author={Dienes, P.},
       title={The {T}aylor series: an introduction to the theory of functions
  of a complex variable},
   publisher={Dover Publications, Inc., New York},
        date={1957},
}

\bib{MR2128719}{book}{
      author={Gasper, George},
      author={Rahman, Mizan},
       title={Basic hypergeometric series},
     edition={Second},
      series={Encyclopedia of Mathematics and its Applications},
   publisher={Cambridge University Press, Cambridge},
        date={2004},
      volume={96},
        ISBN={0-521-83357-4},
         url={https://doi.org/10.1017/CBO9780511526251},
        note={With a foreword by Richard Askey},
}

\bib{MR123484}{article}{
      author={Gordon, Basil},
       title={A combinatorial generalization of the {R}ogers-{R}amanujan
  identities},
        date={1961},
        ISSN={0002-9327},
     journal={Amer. J. Math.},
      volume={83},
       pages={393\ndash 399},
         url={https://doi.org/10.2307/2372962},
}

\bib{GAP4}{manual}{
      author={Group, The~GAP},
       title={{GAP -- Groups, Algorithms, and Programming, Version 4.11.0}},
        date={2020},
         url={\url{https://www.gap-system.org}},
}

\bib{MR1633052}{book}{
      author={Kozen, Dexter~C.},
       title={Automata and computability},
      series={Undergraduate Texts in Computer Science},
   publisher={Springer-Verlag, New York},
        date={1997},
        ISBN={0-387-94907-0},
         url={https://doi.org/10.1007/978-1-4612-1844-9},
}

\bib{MR501091}{article}{
      author={Lepowsky, J.},
      author={Milne, S.},
       title={Lie algebraic approaches to classical partition identities},
        date={1978},
        ISSN={0001-8708},
     journal={Adv. in Math.},
      volume={29},
      number={1},
       pages={15\ndash 59},
         url={https://doi.org/10.1016/0001-8708(78)90004-X},
}

\bib{MR638674}{article}{
      author={Lepowsky, James},
      author={Wilson, Robert~Lee},
       title={A new family of algebras underlying the {R}ogers-{R}amanujan
  identities and generalizations},
        date={1981},
        ISSN={0027-8424},
     journal={Proc. Nat. Acad. Sci. U.S.A.},
      volume={78},
      number={12, part 1},
       pages={7254\ndash 7258},
         url={https://doi.org/10.1073/pnas.78.12.7254},
}

\bib{MR663415}{article}{
      author={Lepowsky, James},
      author={Wilson, Robert~Lee},
       title={A {L}ie theoretic interpretation and proof of the
  {R}ogers-{R}amanujan identities},
        date={1982},
        ISSN={0001-8708},
     journal={Adv. in Math.},
      volume={45},
      number={1},
       pages={21\ndash 72},
         url={https://doi.org/10.1016/S0001-8708(82)80012-1},
}

\bib{MR752821}{article}{
      author={Lepowsky, James},
      author={Wilson, Robert~Lee},
       title={The structure of standard modules. {I}. {U}niversal algebras and
  the {R}ogers-{R}amanujan identities},
        date={1984},
        ISSN={0020-9910},
     journal={Invent. Math.},
      volume={77},
      number={2},
       pages={199\ndash 290},
         url={https://doi.org/10.1007/BF01388447},
}

\bib{MR782227}{article}{
      author={Lepowsky, James},
      author={Wilson, Robert~Lee},
       title={The structure of standard modules. {II}. {T}he case
  {$A^{(1)}_1$}, principal gradation},
        date={1985},
        ISSN={0020-9910},
     journal={Invent. Math.},
      volume={79},
      number={3},
       pages={417\ndash 442},
         url={https://doi.org/10.1007/BF01388515},
}

\bib{MR888628}{article}{
      author={Meurman, A.},
      author={Primc, M.},
       title={Annihilating ideals of standard modules of {${\rm sl}(2,{\bf
  C})^\sim$} and combinatorial identities},
        date={1987},
        ISSN={0001-8708},
     journal={Adv. in Math.},
      volume={64},
      number={3},
       pages={177\ndash 240},
         url={https://doi.org/10.1016/0001-8708(87)90008-9},
}

\bib{NandiPhD}{thesis}{
      author={Nandi, Debajyoti},
       title={Partition identities arising from the standard
  {$A^{(2)}_2$}-modules of level 4},
        type={Ph.D. Thesis},
        date={2014},
}

\bib{1912.03689}{unpublished}{
      author={Rosengren, Hjalmar},
       title={Proofs of some partition identities conjectured by {K}anade and
  {R}ussell},
        note={arXiv:1912.03689},
}

\bib{MR3773946}{incollection}{
      author={Sills, Andrew~V.},
       title={A classical {$q$}-hypergeometric approach to the {$A^{(2)}_2$}
  standard modules},
        date={2017},
   booktitle={Analytic number theory, modular forms and {$q$}-hypergeometric
  series},
      series={Springer Proc. Math. Stat.},
      volume={221},
   publisher={Springer, Cham},
       pages={713\ndash 731},
         url={https://doi.org/10.1007/978-3-319-68376-8_39},
}

\bib{MR3752624}{book}{
      author={Sills, Andrew~V.},
       title={An invitation to the {R}ogers-{R}amanujan identities},
   publisher={CRC Press, Boca Raton, FL},
        date={2018},
        ISBN={978-1-4987-4525-3},
        note={With a foreword by George E. Andrews},
}

\bib{sipser13}{book}{
      author={Sipser, Michael},
       title={Introduction to the theory of computation},
     edition={3},
   publisher={Course Technology},
     address={Boston, MA},
        date={2013},
        ISBN={113318779X},
}

\bib{MR0049225}{article}{
      author={Slater, L.~J.},
       title={Further identities of the {R}ogers-{R}amanujan type},
        date={1952},
        ISSN={0024-6115},
     journal={Proc. London Math. Soc. (2)},
      volume={54},
       pages={147\ndash 167},
         url={https://doi.org/10.1112/plms/s2-54.2.147},
}

\bib{MR2868112}{book}{
      author={Stanley, Richard~P.},
       title={Enumerative combinatorics. {V}olume 1},
     edition={Second},
      series={Cambridge Studies in Advanced Mathematics},
   publisher={Cambridge University Press, Cambridge},
        date={2012},
      volume={49},
        ISBN={978-1-107-60262-5},
}

\bib{2006.02630}{unpublished}{
      author={Takigiku, Motoki},
      author={Tsuchioka, Shunsuke},
       title={Andrews-{G}ordon type series for the level 5 and 7 standard
  modules of the affine {L}ie algebra ${A}^{(2)}_2$},
        note={arXiv:2006.02630},
}

\bib{MR1364151}{article}{
      author={Tamba, Manvendra},
      author={Xie, Chuan~Fu},
       title={Level three standard modules for {$A^{(2)}_2$} and combinatorial
  identities},
        date={1995},
        ISSN={0022-4049},
     journal={J. Pure Appl. Algebra},
      volume={105},
      number={1},
       pages={53\ndash 92},
         url={https://doi.org/10.1016/0022-4049(94)00127-8},
}

\end{biblist}
\end{bibdiv}

\end{document}